\documentclass[journal,twoside,web]{IEEEtran}

\usepackage{color}
\usepackage{graphicx}
\usepackage{amsfonts,amssymb}
\usepackage{bm}
\usepackage{bbm}
\usepackage{amsmath}
\usepackage{algorithm,algorithmic}
\usepackage{mathrsfs}
\usepackage{blkarray}
\usepackage{cases}
\usepackage{arydshln}
\usepackage{epstopdf}
\usepackage{mathrsfs}
\usepackage{dsfont}
\usepackage{cite}
\usepackage[colorinlistoftodos,bordercolor=orange,backgroundcolor=orange!20,linecolor=orange,textsize=scriptsize]{todonotes}
\usepackage{subcaption}
\usepackage{mathtools}
\usepackage{booktabs} 
\usepackage{tikz}
\usetikzlibrary{calc,patterns,decorations.pathmorphing,decorations.markings}
\usetikzlibrary{arrows.meta}
\usetikzlibrary{shapes.geometric, arrows,shadows}
\usepackage{url}
\tikzstyle{node1} = [rectangle, rounded corners, drop shadow,
minimum width=3cm, 
minimum height=1cm,
text centered, 
fill=red!30]
\tikzstyle{node2} = [rectangle, rounded corners, drop shadow,
minimum width=3cm, 
minimum height=1cm,
text centered, 
fill=cyan!30]
\tikzstyle{arrow} = [cyan,thick,->,>=Stealth]
\usepackage{amsthm}
\newtheorem{thm}{Theorem}

\newtheorem{lem}{Lemma}
\newtheorem{rem}{Remark}
\newtheorem{pro}{Proposition}

\newtheorem{defn}{Definition}

\newtheorem{cor}{Corollary}

\DeclareMathOperator{\opd}{d\!}

\DeclareMathOperator{\e}{e}
\DeclareMathOperator{\rank}{rank}

\newcommand{\rev}{\color{black}}
\newcommand{\what}{\widehat}

\newcommand{\p}{\mathbf{P}}
\newcommand{\z}{\mathbf{Z}}

\newcommand{\Qp}{Q_{\mathbf{P}}}
\newcommand{\Qz}{Q_{\z}}
\newcommand{\Rp}{R_{\mathbf{P}}}
\newcommand{\Rz}{R_{\z}}
\newcommand{\Up}{\Upsilon_{\mathbf{P}}}
\newcommand{\Uz}{\Upsilon_{\z}}
\newcommand{\G}{\bm{\widehat{\Sigma}}_G}

\begin{document}
	
	\title{\LARGE \bf
		Time-Domain Moment Matching for Second-Order Systems [Extended Version]
		\thanks{The research leading to the results has received funding from: GAR2023 code 71 Research Grant funded and managed by the Patrimony Foundation (Fundatia "Patrimoniu") of the Romanian Academy, from the Recurrent Fund of Donors, Contract No. 260/28.11.2023; The National Program for Research of the National Association of Technical Universities--GNAC ARUT 2023 grant, contract no. 5/06.10.2023.}}

		\author{ Xiaodong~Cheng, \IEEEmembership{Senior Member,~IEEE} and 
Tudor C. Ionescu
\thanks{X. Cheng is with Mathematical and Statistical Methods Group (Biometris), Wageningen University \& Research,
	6700 AA Wageningen, The Netherlands.
	{\tt\small xiaodong.cheng@wur.nl}}
\thanks{T. C. Ionescu is with the Department of Automatic Control and Systems Engineering, Politehnica University of Bucharest, 060042 Bucharest and with the "Gheorghe-Mihoc--Caius Iacob" Institute of Statistical Mathematics and Applied Mathematics of the Romanian Academy, 050711 Bucharest, Romania. E-mail: {\tt tudor.ionescu@upb.ro; tudor.ionescu@ismma.ro}}
}

\maketitle
\vspace*{-13pt}

\begin{abstract}
The paper develops a second-order time-domain moment matching framework for the structure-preserving model reduction of high-dimensional second-order dynamical systems, avoiding the first-order double-sized equivalent representation. The moments of a second-order system are characterized by the solutions of second-order Sylvester equations, leading to families of parameterized second-order reduced models that match the moments of the original system at selected interpolation points. A two-sided moment matching problem is also addressed, yielding a unique second-order reduced system that matches two distinct sets of interpolation points. Furthermore, we construct reduced second-order systems that match the moments of both the transfer function and its first-order derivative. Then, we also discuss how the proposed framework can be extended to multiple-input multiple-output (MIMO) second-order systems through tangential interpolation, and we identify the main open difficulties in extending the derivative-matching and pole-zero placement results to the MIMO setting. The theory is illustrated on a numerical example of vibrating systems.
\end{abstract}

\section{Introduction}
Second-order dynamical systems are commonly used to capture the behavior of various physical systems, such as electrical circuits, power systems, mechanical systems, see, e.g., \cite{yan2008second,Dorier2014PowerNetworks,Safaee2021SecondOrder,Morzfeld2010Vibration,Koutsovasilis2008comparison}. In this work, we focus on the dynamics of a single-input and single-output (SISO)  linear time-invariant   second-order system described by
\begin{equation} \label{sys:2o}
\bm{\Sigma}:
\begin{cases} 
	M\ddot{x}(t) + D\dot{x}(t) + Kx(t) & = Bu(t),
	\\ 
	C_1 \dot{x}(t) + C_0  x(t) & =     y(t),
\end{cases}
\end{equation}
with $x(t)\in\mathbb{R}^n,\ u(t)\in\mathbb{R},\ y(t)\in\mathbb{R}$, $\forall t\geq 0$, where $M, D, K \in \mathbb{R}^{n \times n}$ are commonly referred to as the mass, damping, and stiffness matrices in mechanical systems. $B \in \mathbb{R}^{n}$ is the input matrix of external forces, and $C_0, C_1 \in \mathbb{R}^{1 \times n}$ are the output matrices for positions and velocities. The transfer matrix of the system $\bm{\Sigma}$ is given by
\begin{equation*}
W(s)= (C_1s + C_0)(Ms^2 + Ds + K)^{-1}B,
\end{equation*}
with $2n$ finite poles in the symmetric\footnote{A set of complex numbers is symmetric if and only if, for any element in the set, the complex-conjugate counterpart is also in the set, including multiplicities.} set
\begin{equation} \label{Omega}
\Omega: = \{s \in \mathbb{C} \mid \det(Ms^2 + Ds +K) =0 \},
\end{equation}
with $|\Omega| = 2n$.\\
In real applications, the model description \eqref{sys:2o} often has a high dimension $n$, requiring a large amount of computational resources and thus hindering simulation, prediction, and control of such systems.
Therefore, model reduction techniques for second-order dynamical systems of high dimensions have received increasing attention, and reduced order models are indispensable for efficient analysis and optimization of the structured systems. 


The essential problem in model reduction of second-order systems is the preservation of the second-order structure, allowing for a physical interpretation of the resulting approximation. However, a direct application of structure preservation method e.g., \cite{cheng2023optimal} is not necessarily straightforward. Although a second-order system \eqref{sys:2o} can be rewritten in first-order form, yielding the first-order equivalent system, with state vector $[x(t)^\top \  \dot{x}(t)^\top]^\top$, reduced via first-order reduction methods, the resulting reduced-order models typically destroy the second-order structure.  
To preserve the second-order structure, second-order balancing methods have been proposed in, e.g., \cite{meyer1996balancing2o,chahlaoui2006balancing2o,reis2008BT2O,benner2011efficient2O,benner2013improved2O}. The so-called \textit{position and velocity Gramians} are defined as the diagonal blocks in the Gramian matrices of the first-order representation. Then, balanced truncation is performed based on different pairs of position and velocity Gramians. However, unlike the balanced truncation for first-order stable systems, these methods do not generally preserve stability or provide global error bounds. A port-Hamiltonian approach in \cite{hartmann2010BT2O} reduces a second-order system via a generalized Hamiltonian framework and preserves the Hamiltonian structure and stability. The model reduction problem in \cite{Sato2017Riemannian,Yu2021h_2} is tackled by optimization approaches, where reduced systems are constructed as the optimal solution of an $H_2$-optimization problem subject to certain structural constraints.
In \cite{XiaodongTAC20172OROM,cheng2016secondclustering,XiaodongACOM2018Power,Ishizaki2015clustered}, a clustering-based framework is considered to simplify the structure of second-order network systems, and the scheme is based on identifying and aggregating nodal states that have similar responses to external input signals. 

Moment matching techniques provide efficient tools for model reduction of dynamical systems, see 
\cite{antoulas2005approximation,gallivan2004sylvester,antoulas2010interpolatory,Astolfi2010MM,astolfi2020model} for an extensive overview for first-order systems. Using Krylov subspace projection matrices, reduced models are constructed to match the original system at selected interpolation points in the complex plane. Recent extensions to second-order systems are in, e.g.,
\cite{bai2005Arnoldi2o,salimbahrami2order,beattie2005Krylov2o,qiu2018interpolatory2O,Vakilzadeh2018krylov}, with second-order Krylov subspaces introduced to preserve second-order structure. In \cite{beeumen-nimmen-lombaert-meerbergen-IJNME2012}, a Krylov-based approach to the model reduction of second-order systems with structural damping and quadratic output is presented.

A time-domain approach to moment matching has been presented in \cite{Astolfi2010MMCDC,Astolfi2010MM}, where the moments of a system are characterized by the unique solutions of Sylvester equations. There is a one-to-one relation between the moments and the steady-state response of the system at the selected frequencies, i.e., the interpolation points. This time-domain approach has been further developed in e.g., \cite{Ionescu2013AUT,Ionescu2016TwoSided,Ionescu2014SCL} for port-Hamiltonian systems and two-sided moment matching problems.

The paper extends the time-domain moment matching approach to linear second-order systems in \eqref{sys:2o}, yielding a so-called second-order time-domain moment matching framework. Particularly, we represent the moments of $W(s)$ at a set of interpolation points by the unique solution of a second-order Sylvester equation. Thereby, a family of parameterized second-order reduced models is constructed. Using the set of free parameter matrices, we calculate the approximations that preserve stability and passivity. Another contribution is the two-sided second-order moment matching approach, where the second-order approximating model matches the moments of $W(s)$ at two distinct sets of interpolation points. Furthermore, we also study the problem of second-order time-domain moment matching for the first-order derivative of the transfer function of the system \eqref{sys:2o}, $W^\prime(s)=\opd W(s)/\opd s$. The moments are shown to have a one-to-one relation with the steady-state response of the system composed of the state-space representation of $W^\prime(s)$ and two dual signal generators in a cascade form, one exciting the input with signals at the selected frequencies and the other one modulating the resulting outputs. We present a reduced-order model achieving second-order moment matching at both zero and first-order derivatives of $W(s)$. For MIMO systems, interpolation conditions must be formulated in terms of
right and left tangential data rather than scalar transfer-function values.
This distinction is essential when extending time-domain moment matching to
MIMO systems, especially for complex interpolation data and bitangential
interpolation conditions.

\paragraph*{Contributions w.r.t. \cite{cheng-i-iftime-necoara-CDC2024}} Preliminary results have been presented in \cite{cheng-i-iftime-necoara-CDC2024}, without proofs. In this paper, we provide a detailed, systematic framework of second-order time-domain moment matching. Furthermore, we prove that the second-order Sylvester equations involved indeed have unique solutions. The notion of the moment is herein related to the steady-state response of a second-order system driven/driving a (generalized) signal generator defined by the interpolation frequencies. The second-order time-domain moment matching of the derivative of the second-order transfer function is derived.  The main development is carried out for SISO systems, where the derivative
matching and pole-zero placement results can be stated in complete form. In addition, we include an initial MIMO extension based
on tangential interpolation, including one-sided and two-sided reduced
second-order families and a bi-tangential Hermite interpolation property.

The paper is organized as follows. In Section~\ref{sect_prel}, we present
preliminary results regarding time-domain moment matching for linear systems. In Section~\ref{sec:mmsecord},  the moments of second-order systems are characterized with second-order Sylvester equations, and the time-domain moment matching approach for second-order systems is presented. The moment matching problems pertaining to two-sided moment matching, pole placement, and first-order derivatives are discussed in Section~\ref{sec:twosided}. Section~\ref{sec:mimo_extension} discusses the extension toward MIMO
second-order systems using tangential interpolation and identifies the main
open difficulties. Section~\ref{sec:Example} illustrates the proposed approaches using a mass-spring-damper system, and finally, concluding remarks are made in Section~\ref{sec:Conclusion}. The technical proofs of the results are found in the Appendix.
\paragraph*{Notation} $\mathbb{R}$ and $\mathbb{C}$ denote the sets of real and complex numbers, respectively. $\mathbb{C}^-$ and $\mathbb{C}^0$ are the sets of complex numbers with negative real part and zero real part, respectively.  $\emptyset$ is the empty set, and $\mathds{1}_{n \times m}$ represents a matrix with all elements equal to $1$.
For a matrix $A \in \mathbb{R}^{n \times m}$, $A^\top \in \mathbb{R}^{m \times n}$ denotes the transpose of $A$, respectively. $\sigma(A)$ represents the set of the eigenvalues of $A$,  and $\det(A)$ represents the determinant of $A$. Moreover, $A^\dagger$ is the left pseudoinverse of $A$. For a real symmetric matrix $X$, $X > 0$ ($X< 0$) means that $X$ is positive definite (negative definite).

\section{Preliminaries}\label{sect_prel}
\noindent In this section, we recall the notion of time-domain moment matching a stable LTI system of order one, see e.g., \cite{astolfi-TAC2010,i-astolfi-colaneri-SCL2014}.

\subsection{Time-Domain Moment Matching for Linear  Systems}

\noindent Consider a single-input single-output (SISO) linear time-invariant  (LTI) minimal  system
\begin{equation}\label{system}
\begin{split}
	&\Sigma: \;\; \dot x = Ax+Bu, \quad y=Cx,
\end{split}
\end{equation}
with the state $x\in\mathbb{R}^n$, the input $u\in\mathbb{R}$ and the output $y\in\mathbb R$. The transfer function of \eqref{system} is
\begin{equation}\label{tf}
K(s)=C(sI-A)^{-1}B,\quad K:\mathbb{C} \to \mathbb {C}.
\end{equation}
%


The moments \footnote{The term `moment matching' is also used in the rational covariance extension problem. Therein, the moments correspond to given covariance lags of a stochastic process, and the objective is to extend a partial covariance sequence to one consistent with a rational spectral density, see e.g., \cite{byrnes1998convex,byrnes-lindquist-SIAM2008}. } of (\ref{tf}) are defined as follows.
\begin{defn}\label{def_moment}\cite{antoulas-2005,astolfi-TAC2010}
The $k$-moment
of system (\ref{system}) with the transfer function $K$ as in \eqref{tf}, at $s_{1}\in\mathbb C$ is defined by
$$\eta_{k}(s_{1})={(-1)^{k}}/{k!}\left[{\opd^{k}K(s)}/{\opd s^{k}}\right]_{s=s_1}\in\mathbb C.$$
\end{defn}

Pick the symmetric set $\{s_1,\dots, s_\nu\}\subseteq\mathbb C\setminus\sigma(A)$, and let $S\in\mathbb{R}^{\nu\times\nu}$, such that $\sigma(S)=\{s_1,\dots,s_\nu\}$. Let $L\in\mathbb{R}^{1\times\nu}$, such that the pair $(L,S)$ is observable. Let $\Pi\in\mathbb{R}^{n\times\nu}$ be the solution of the Sylvester equation
\begin{equation}\label{eq_Sylvester_Pi}
A\Pi+BL  =  \Pi S.
\end{equation}
Since \eqref{system} is minimal and $\sigma(A)\cap\sigma(S)=\emptyset$, then
$\Pi$ is the unique solution of the equation (\ref{eq_Sylvester_Pi})
and ${\rm rank}\ \Pi=\nu$, see e.g. \cite{desouza-bhattacharyya-LAA1981}.
Then, the moments of (\ref{system}) are characterized as follows.


\begin{pro}
\label{prop_mom_time}
\cite{astolfi-TAC2010}
\label{def_PI}
The moments of system (\ref{system}) at the interpolation points
$\{s_{1},s_{2},...,s_{\nu}\}=\sigma(S)$, such that $\sigma(S)\cap\sigma(A)=\emptyset$, are in one-to-one relation\footnote{By a one-to-one relation between a set of moments and the elements of a matrix, we mean that, for a fixed realization of the signal generator, the moments are uniquely determined by the elements of the matrix, and vice versa. Under a similarity transformation of the signal generator, the representing matrix changes accordingly, while the corresponding moments remain invariant.} with the elements of the matrix $C\Pi$.
\end{pro}

\noindent The following proposition gives necessary and sufficient conditions
for a $\nu$-order system to achieve moment matching.

\begin{pro}
\cite{astolfi-TAC2010}
\label{prop_FGL}
Consider the LTI system
\begin{equation}
	\label{red_mod_F}
	\begin{split}
		&\dot{\xi}=F\xi+Gu,\quad \psi=H\xi,
	\end{split}
\end{equation}
with $\xi(t)\in\mathbb R^\nu,\ \forall t\geq0,\ F\in\mathbb{R}^{\nu\times\nu},\ G\in\mathbb{R}^{\nu}$ and $ H\in\mathbb{R}^{1\times\nu}$, and the corresponding  transfer function
$
K_G(s)=H(sI-F)^{-1}G.
$
Fix $S\in\mathbb{R}^{\nu\times\nu}$ and $L\in\mathbb{R}^{1\times\nu}$,
such that the pair $(L,S)$ is observable and $\sigma(S)\cap\sigma(A)=\emptyset$.
The system (\ref{red_mod_F}) matches the moments of (\ref{system}) at $\sigma(S)$ if and only if
\begin{align*}\label{eq_MM_CPi}
	HP&=C\Pi, \quad \sigma(F)\cap\sigma(S)=\emptyset.
\end{align*}
where $P\in\mathbb{R}^{\nu\times\nu}$ is any invertible matrix uniquely satisfying the Sylvester equation $FP+GL=PS$.
\end{pro}

Throughout the rest of the paper, we consider $\nu<n$. Then \eqref{red_mod_F} is a reduced order model of \eqref{system} matching $\nu$ moments at $\sigma(S)$. We are now ready to present a family of $\nu$ order models parameterized in $G$ that match $\nu$ moments of the given system \eqref{system} at $\sigma(S)$. The reduced system
\begin{equation}\label{redmod_CPi}
\Sigma_{G}: \,\dot{\xi}=(S-GL)\xi+Gu,\quad \psi=C\Pi\xi,
\end{equation}
with the transfer function
\begin{equation}\label{tf_redmod_CPi}
K_G(s)=C\Pi(sI-S+GL)^{-1}G,
\end{equation}
describes a family of $\nu$ order models that achieve moment matching at $\sigma(S)$ \emph{fixed}, i.e,
\begin{enumerate}
\item $ \Sigma_{G}$ matches the moments $C\Pi$ of \eqref{system} $\forall G\in\mathbb R^{\nu}$,
\item $\sigma(S-GL) \cap \sigma(S)=\emptyset$.
\end{enumerate}

\section{Moments and Moment Matching of Second-order System}
\label{sec:mmsecord}
In this section, we lay out the time-domain moment matching framework in the case of second-order systems \eqref{sys:2o} with the transfer function $W(s)$.

\subsection{Moments of Second-Order Systems}
\label{sec:moments}
In this section, we characterize the moments of the second-order system $\bm{\Sigma}$ in \eqref{sys:2o} at a set of interpolation points different from the poles of $\bm{\Sigma}$. 


Following \cite{antoulas2005approximation,Astolfi2010MM}, the moments of a second-order system \eqref{sys:2o} are defined as follows.
\begin{defn} \label{def:moments}
Let $s_\star \in \mathbb{C}$ such that $s_\star \notin \Omega$. The 0-moment of $W(s)$ at $s_\star \in \mathbb{C}$ is the complex number 
\begin{equation*} \label{eq:0-moment}
	\eta_0(s_\star) = W(s_\star) = (C_1s_\star  + C_0) ( Ms_\star^2 + Ds_\star + K)^{-1}B,
\end{equation*}
and the $k$-moment at $s_\star \in \mathbb{C}$ is defined by
\begin{equation} \label{eq:k-moment}
	\eta_k(s_\star)=\frac{(-1)^k}{k!} 
	\left[ \frac{\opd^k}{\opd s^k}W(s) \right]_{s=s_\star}, \ k \geq 1 \ \text{and integer}.
\end{equation}
\end{defn}

Note that the 0-moment of $W(s)$ at $s_\star$ can be written as $\eta_0(s_\star) =  C_0 \Pi + C_1 \Pi s_\star$, where $\Pi$ is the unique solution of
the matrix equation 
$
M \Pi s_\star^2 + D \Pi s_\star + K \Pi = B.
$  
Then, the following lemma is obtained for moments at distinct interpolation points.
\begin{lem} \label{lem:2Omoments1}
Let
\begin{align*}
	S &= \mathrm{diag}(s_1, s_2, \cdots, s_\nu), \ \text{and} \ L = \begin{bmatrix}
		l_1 & ... & l_\nu
	\end{bmatrix}, \\
	Q &= \mathrm{diag}(s_{\nu+1}, s_{\nu+1}, \cdots, s_{2\nu}),\ \text{and} \ R = \begin{bmatrix}
		r_1  & ... & r_\nu 
	\end{bmatrix}^\top,
\end{align*}
where $s_i \in \mathbb{C} \setminus \Omega$, $\forall i = 1,2, \cdots, 2\nu$, $l_i \in \mathbb{C}$, and $r_i \in \mathbb{C}$, $i = 1, 2, ..., \nu$. With the pair $(L,S)$ observable, and $(Q,R)$ controllable, the 0-moments $\eta_0(s_i)$ satisfy 
\begin{align*}
	\begin{bmatrix}
		\eta_0(s_1) & \eta_0(s_2) & \cdots & \eta_0(s_\nu)
	\end{bmatrix} &= C_0 \Pi + C_1 \Pi S, \\
	\begin{bmatrix}
		\eta_0(s_{\nu+1}) & \eta_0(s_{\nu+2}) & \cdots & \eta_0(s_{2\nu})
	\end{bmatrix} &= \Upsilon B,
\end{align*}
where $\Pi$, $\Upsilon \in \mathbb{R}^{n \times \nu}$ satisfy the following second-order Sylvester equations
\begin{subequations}
	\begin{align}
		\label{eq:Sylv_nu}
		M \Pi S^2  + D \Pi S + K\Pi &= BL, \\
		\label{eq:Sylv_nu2}
		Q^2 \Upsilon M   +   Q \Upsilon D + \Upsilon K  &= R C_0 + Q R C_1.
	\end{align}
\end{subequations}
\end{lem}


The proof of Lemma~\ref{lem:2Omoments1} is given in Appendix~\ref{app:Lem:2Omoments1}. Furthermore, the characterization of the moments at a single interpolation point with higher-order derivatives is provided in the following lemma. 
\begin{lem} \label{lem:2Omoments2}
Consider \eqref{sys:2o} and $s_\star, z_\star \in \mathbb{C} \setminus \Omega$. 
Let the matrices $S \in \mathbb{R}^{(\nu +1) \times (\nu +1)}$,
$L \in \mathbb{R}^{1 \times (\nu +1)}$ and $Q \in \mathbb{R}^{(\nu +1) \times (\nu +1)}$, $R \in \mathbb{R}^{(\nu +1)\times 1}$ be such that the pair $(L, S)$ is observable, and the pair $(Q,R)$ is controllable, respectively. Suppose $S$ and $Q$ are non-derogatory\footnote{A matrix is called non-derogatory if its minimal and characteristic polynomials are identical.} such that 
\begin{equation*}
	\det(s I - S) = (s - s_\star)^{\nu +1}, \ \det(s I - Q) = (s - z_\star)^{\nu +1}.
\end{equation*}
Then the following statements hold.
\begin{enumerate}
	\item There exists a one-to-one relation between the moments $\eta_0(s_\star)$, $\eta_1(s_\star)$, $\cdots$, $\eta_\nu(s_\star)$ 
	and the matrix $C_0 \Pi + C_1 \Pi S$,
	where $\Pi$ satisfies
	\begin{equation} \label{eq:Sylv_k}
		M \Pi S^2  + D \Pi S + K\Pi = BL.
	\end{equation} 
	
	\item There exists a one-to-one relation between the moments $\eta_0(z_\star)$, $\eta_1(z_\star)$, $\cdots$, $\eta_\nu(z_\star)$ and 
	and the matrix $\Upsilon B$,
	where $\Upsilon$ satisfies
	\begin{equation} \label{eq:Sylv_k2}
		Q^2 \Upsilon M   +   Q \Upsilon D + \Upsilon K  = R C_0 + Q R C_1.
	\end{equation} 
\end{enumerate}
\end{lem}

The detailed proof is given in Appendix~\ref{app:Lem:2Omoments2}. Now, let us combine the conclusions in Lemma \ref{lem:2Omoments1} and Lemma \ref{lem:2Omoments2}, which leads to the following result.
\begin{thm} \label{thm:Moments}
Consider the second-order system \eqref{sys:2o} with transfer function $W(s)$. Let the matrices $S \in \mathbb{R}^{\nu \times \nu}$,
$L \in \mathbb{R}^{1 \times \nu}$ and $Q \in \mathbb{R}^{\nu \times \nu}$, $R \in \mathbb{R}^{\nu}$ be such that the pair $(L, S)$ is observable, and the pair $(Q,R)$ is controllable, respectively. 
Then, the following statements hold.
\begin{enumerate}
	\item  If $\sigma(S) \cap \Omega = \emptyset$, there is a one-to-one relation between the moments of $W(s)$ at $\sigma(S)$
	and the matrix 
	$C_0 \Pi + C_1 \Pi S$, where $\Pi \in \mathbb{R}^{n \times \nu}$ is the unique
	solution of  
	\begin{equation} \label{eq:Sylv1}
		M \Pi S^2  + D \Pi S + K\Pi = BL.
	\end{equation}
	
	\item  If $\sigma(Q) \cap \Omega = \emptyset$, there is a one-to-one relation between the moments of $W(s)$ at $\sigma(Q)$
	and the matrix 
	$\Upsilon B$, where $\Upsilon \in \mathbb{R}^{\nu \times n}$ is the unique
	solution of  
	\begin{equation} \label{eq:Sylv2}
		Q^2 \Upsilon M   +   Q \Upsilon D + \Upsilon K  = R C_0 + Q R C_1.
	\end{equation}
\end{enumerate}
\end{thm}
The proof of the result is presented in Appendix~\ref{app:thm:Moments}. 
\begin{rem}
The proof of Theorem~\ref{thm:Moments} provides an effective way to obtain $\Pi$ and $\Upsilon$ as the solutions of the second-order Sylvester equations in \eqref{eq:Sylv1} and \eqref{eq:Sylv2}. We can compute the first-order Sylvester equations in \eqref{eq:Sylv1o1} and \eqref{eq:Sylv2o1} to obtain $\widetilde{\Pi}$ and $\widetilde{\Upsilon}$, from which, $\Pi$ and $\Upsilon$ can be uniquely determined. 
\end{rem}

Throughout the rest of the manuscript, we make the working assumption that all the solutions of the second-order Sylvester equations \eqref{eq:Sylv1} and \eqref{eq:Sylv2}, respectively, have full rank.

\subsection{Moment matching-based Reduced-Order Second-Order Systems}

Using the characterization of moments in Theorem~\ref{thm:Moments}, we now define the families of second-order reduced models achieving moment matching at the given interpolation points. The following results are necessary and sufficient
conditions for a low-order system $\bm{\widehat{\Sigma}}$ to achieve moment matching.

\begin{pro} \label{pro:family}
Consider the second-order reduced model 
\begin{equation*} \label{sys:2or}
	\bm{\widehat{\Sigma}}:
	\begin{cases}
		F_2 \ddot{\xi}(t) + F_1 \dot{\xi}(t) + F_0 \xi(t) & = G u (t), \\
		H_1 \dot{\xi}(t) + H_0 \xi(t) & = \psi(t),
	\end{cases}
\end{equation*}
with $\xi(t), \dot{\xi}(t) \in \mathbb{R}^\nu$, $\psi(t)\in\mathbb R$, $F_i \in \mathbb{R}^{\nu \times \nu}$, for $i = 0,1,2$, and $G \in \mathbb{R}^{\nu}$, $H_1, H_0 \in \mathbb{R}^{1 \times \nu}$. Denote the following (symmetric) set 
\begin{equation} \label{eq:hatOmega}
	\widehat{\Omega}: = \{s \in \mathbb{C}: \det(s^2 F_2 + s F_1 + F_0) =0 \},\ |\widehat\Omega| = 2\nu.
\end{equation}
Let $S \in \mathbb{R}^{\nu \times \nu}$,
$L \in \mathbb{R}^{1 \times \nu}$ and $Q \in \mathbb{R}^{\nu \times \nu}$, $R \in \mathbb{R}^{\nu}$ be such that the pair $(L, S)$ is observable, and the pair $(Q,R)$ is controllable, respectively. 
\begin{enumerate}
	\item Assume that 
	$
	\sigma(S) \cap \Omega = \emptyset \ \text{and} \ \sigma(S) \cap \widehat{\Omega} = \emptyset
	$. The reduced system $\bm{\widehat{\Sigma}}$ matches the moments of $\bm{\Sigma}$ at $\sigma(S)$ if and only
	\begin{equation*}\label{eq:MM2O}
		C_0 \Pi + C_1 \Pi S = H_0 P + H_1 P S
	\end{equation*}
	where $P \in \mathbb{R}^{\nu \times \nu}$ is unique solution of the second-order Sylvester equation
	\begin{equation*} \label{eq:Sylv-F}
		F_2 P S^2  + F_1 P S + F_0 P = G L.
	\end{equation*}
	
	\item Assume that 
	$
	\sigma(Q) \cap \Omega = \emptyset \ \text{and} \ \sigma(Q) \cap \widehat{\Omega} = \emptyset
	$.   The reduced system $\bm{\widehat{\Sigma}}$ matches the moments of $\bm{\Sigma}$ at $\sigma(Q)$ if and only if
	\begin{equation}\label{eq:MM2O2}
		\Upsilon B = P G,
	\end{equation}
	where $P \in \mathbb{R}^{\nu \times \nu}$ is unique solution of the second-order Sylvester equation
	\begin{equation*} \label{eq:Sylv-F2}
		Q^2 P F_2   +   Q P F_1 + P F_0  = R H_0 + Q R H_1.
	\end{equation*}
\end{enumerate}
\end{pro}
The proof follows a similar reasoning as in \cite{Astolfi2010MM,Ionescu2016TwoSided}. Choosing $P = I_\nu$, we then obtain the family of
second-order reduced model $\bm{\widehat{\Sigma}}_G:$
\begin{equation} \label{sys:2orG}
\begin{cases}
	F_2 \ddot{\xi} + F_1 \dot{\xi} + (GL - F_2 S^2 - F_1 S) \xi & = G u , \\
	H_1 \dot{\xi} + (C_0\Pi + C_1\Pi S  - H_1S) \xi & = \psi(t),
\end{cases}
\end{equation}
parameterized by $F_1$, $F_2$, $G$ and $H_1$, and matches the moments of $\bm\Sigma$ at $\sigma(S)$. Analogously, the reduced model $\bm{\widehat{\Sigma}}_H$:
\begin{equation}\label{sys:2orH}
\begin{cases}
	F_2 \ddot{\xi} + F_1 \dot{\xi} + (RH_0 + Q R H_1 - Q^2 F_2 - Q F_1) \xi &=  \Upsilon B u , \\
	H_1   \dot{\xi} + H_0   \xi & =  \psi(t),
\end{cases}
\end{equation}
parameterized by $F_1, F_2$, $H_0$, and $H_1$, matches the moments of $\bm\Sigma$ at $\sigma(Q)$.  

Throughout the rest of the manuscript, we assume that the resulting models $\bm{\widehat{\Sigma}}$, computed with the selected data, are such that $\det(s^2 F_2 + s F_1 + F_0)\not\equiv 0$ and $\widehat\Omega$ is finite, with $|\widehat\Omega| = 2\nu$. If the constraint is, by chance, not satisfied, then the data are altered to satisfy the constraint. 

\subsection{Stability and Passivity Preserving Moment Matching}
\label{sec:stabilitypassivity}
Based on the families of $\nu$-order models in \eqref{sys:2orG} and \eqref{sys:2orH},  we derive second-order $\nu$-dimensional models that not only match the moments of the original system $\bm{\Sigma}$ at a prescribed set of finite interpolation points but also preserve stability and passivity of $\bm{\Sigma}$.

The second-order system $\bm{\Sigma}$ in \eqref{sys:2o} is asymptotically stable if $
M > 0, \ D > 0, \ \text{and} \ K > 0 
$ \cite{bernstein1995second}.
It immediately leads to the following results.
\begin{enumerate}
\item 
The second-order reduced system $\bm{\widehat{\Sigma}}_G$ is asymptotically stable for any $G$, $F_2 > 0$,  and $F_1>0$ that satisfy
\begin{equation} \label{eq:stabcond1}
	GL - F_2 S^2 - F_1 S > 0.
\end{equation}

\item  The second-order reduced system $\bm{\widehat{\Sigma}}_H$ is asymptotically stable for any  $H_0$, $H_1$, $F_2 > 0$,  and $F_1>0$ that satisfy
\begin{equation}\label{eq:stabcond2}
	RH_0 + Q R H_1 - Q^2F_2  - Q F_1 >0.
\end{equation}
\end{enumerate} 

Note that both \eqref{eq:stabcond1} and \eqref{eq:stabcond2} are linear matrix inequalities (LMIs), which are computed via standard LMI solvers, e.g, YALMIP and CVX. Furthermore, with free parameters $F_1, F_2 > 0$, and $G \in \mathbb{R}^{\nu}$, there always exists a solution for \eqref{eq:stabcond1}. Similarly, with $F_1, F_2 > 0$, and $H_1, H_0\in \mathbb{R}^{1 \times \nu}$, a solution for \eqref{eq:stabcond2} is also guaranteed. Thereby, we present a particular choice of these parameters in a special case.
\begin{pro}
Consider $S$ and $Q$ with negative real eigenvalues such that
$
S = T^{-1} \Lambda_S T, \ Q = Z \Lambda_Q Z^{-1},
$ 
with $\Lambda_S, \Lambda_Q < 0$ diagonal and $T, Z$ nonsingular.
\begin{enumerate}
	\item Let 
	$
	F_1 = T^\top D_S T, \ F_2 = T^\top E_S T, \ G = L^\top,
	$
	with diagonal matrices $D_S>0$ and $E_S$ such that $0<E_S<-D_S\Lambda_S^{-1}$.
	Then, the reduced system $\bm{\widehat{\Sigma}}_G$ is asymptotically stable.
	\item Let 
	$
	F_1 = Z D_Q Z^\top, \ F_2 = Z E_Q Z^\top, \ H_0 = R^\top, \ H_1 = R^\top Q^\top
	$
	with diagonal matrices $D_Q>0$ and $E_Q$ such that $0<E_Q<-D_Q\Lambda_Q^{-1}$.
	Then, the second-order reduced system $\bm{\widehat{\Sigma}}_H$ is asymptotically stable.
\end{enumerate}
\end{pro} 
\begin{proof}
Since $\Lambda_S<0$, the matrix
$P_S:=-D_S\Lambda_S^{-1}-E_S$ is diagonal and positive definite.
Thus, $F_1=T^\top D_S T$ and $F_2=T^\top E_S T$ are positive definite. Moreover,
\begin{align*}
	GL-F_2S^2-F_1S
	&=L^\top L-T^\top E_S\Lambda_S^2 T-T^\top D_S\Lambda_S T  \\
	&=L^\top L+T^\top P_S\Lambda_S^2 T>0.
\end{align*}
The strict inequality follows from $P_S\Lambda_S^2>0$ and $L^\top L\geq 0$. Hence, condition \eqref{eq:stabcond1} holds, and Proposition~4 implies the first statement.

For the second statement, define $P_Q:=-D_Q\Lambda_Q^{-1}-E_Q>0$. Then, $F_1=ZD_QZ^\top$ and $F_2=ZE_QZ^\top$ are positive definite. Furthermore,
\begin{align*}
	&RH_0+QRH_1-Q^2F_2-QF_1  \\
	&=RR^\top+QRR^\top Q^\top
	+Z\left(-E_Q\Lambda_Q^2-D_Q\Lambda_Q\right)Z^\top \\
	&=RR^\top+QRR^\top Q^\top+ZP_Q\Lambda_Q^2Z^\top>0.
\end{align*}
Therefore, condition \eqref{eq:stabcond2} holds, and the reduced system $\bm{\widehat{\Sigma}}_H$ is asymptotically stable.
\end{proof}

Next, a passivity-preserving model reduction for the second-order system $\bm{\Sigma}$ is discussed. It follows from e.g., \cite{hartmann2010BT2O,XiaodongTAC20172OROM} that  
the original system $\bm{\Sigma}$ is passive if 
\begin{equation} \label{eq:passM}	
M > 0, \ D > 0, \ K > 0
, \	
C_1 = B^\top, \ \text{and} \ C_0 = 0.
\end{equation}
Then, the following results hold.
\begin{pro} \label{pro:pass}
Consider the original second-order system $\bm{\Sigma}$, which satisfies the passivity condition in \eqref{eq:passM}. The second-order reduced system $\bm{\widehat{\Sigma}}_G$ is passive if $G^\top = H_1 = C_1 \Pi$, and $F_1, F_2>0$ satisfy
\begin{equation} \label{eq:passcond1}
	(\Pi^\top M \Pi - F_2) S^2 + (\Pi^\top D \Pi - F_1) S + \Pi^\top K \Pi > 0.
\end{equation}
Moreover, the second-order reduced system $\bm{\widehat{\Sigma}}_H$ is passive if $H_0 = 0$, $H_1 = B^\top \Upsilon^\top $, and $F_1, F_2>0$ satisfy 
$		Q^2(\Upsilon M \Upsilon^\top - F_2)   + Q (\Upsilon D \Upsilon^\top - F_1) + \Upsilon K \Upsilon^\top > 0.
$
\end{pro}
\begin{proof}
As the conditions $G^\top = H_1$, $F_1, F_2>0$ are given, to show the passivity of  $\bm{\widehat{\Sigma}}_G$, we only need the positive definiteness of $F_0$, namely
$GL - F_2 S^2 - F_1 S > 0$. By \eqref{eq:passcond1}, we have 
$
\Pi^\top BL - F_2 S^2 - F_1 S  > 0
$, which holds since $B = C_1$.
The proof for $\bm{\widehat{\Sigma}}_H$ follows similar reasoning.
\end{proof}

Based on the above results, the following result is yielded.
\begin{pro} \label{prop:passivity} 
Consider the second-order system $\bm{\Sigma}$ asymptotically stable and satisfying the passivity condition in \eqref{eq:passM}. The second-order reduced system $\bm{\widehat{\Sigma}}_G$ with parameters
\begin{align*}  
	F_2 &= \Pi^\top M \Pi, F_1 = \Pi^\top D \Pi, F_0 = \Pi^\top K \Pi, \nonumber \\
	G & = \Pi^\top B, H_1 = B^\top \Pi. 
\end{align*}
and reduced system $\bm{\widehat{\Sigma}}_H$ with parameters
\begin{align*}  
	F_2 &= \Upsilon M \Upsilon^\top, F_1 = \Upsilon D \Upsilon^\top, F_0 = \Upsilon K \Upsilon^\top, \nonumber \\
	G & = \Upsilon B, H_1 = B^\top \Upsilon^\top. 
\end{align*}
are asymptotically stable and passive.
\end{pro}
\begin{rem}
Passivity preservation in Proposition~\ref{prop:passivity} is the direct result of Galerkin projections with full rank matrices $\Pi$ and $\Upsilon$, which preserve positive definiteness of the reduced matrices $F_2$, $F_1$, and $F_0$. Thereby, the reduced model retains a meaningful physical interpretation in terms of mass, damping, and stiffness.
\end{rem}
%
%



\section{Two-Sided Moment Matching}
\label{sec:twosided}
This section presents a two-sided second-order time-domain	moment matching approach to obtain a unique
$\nu$-order model matching both the moments of \eqref{sys:2o} at interpolation points in two distinct sets $\sigma(S)$ and $\sigma(Q)$, simultaneously.


Consider two signal generators as follows.
\begin{align} \label{geneSL}
\dot{\omega} = S \omega, \ \omega(0) \neq 0, \ \theta = L \omega, 	
\end{align} 
and  
\begin{align}\label{geneQR}
&\dot{\varpi} = Q \varpi + R \psi, \varpi(0) = 0, \nonumber \\ 
&d = \varpi + (Q\Upsilon M + \Upsilon D - RC_1) x + \Upsilon M \dot{x},
\end{align}
where $\omega, \varpi \in \mathbb{R}^\nu$. {\rev Following  \cite{scarciotti-CDC2015generators,padoan2017geometric}, we assume the minimality of the triple $(L, S, \omega(0))$ for  the signal generator  \eqref{geneSL}, which implies the observability of the pair $(L, S)$ and the excitability of the pair  $(S, \omega(0))$, or equivalently, the controllability of the system $\dot{\widetilde{\omega}} = S \widetilde{\omega} + \omega_0 u$. Furthermore, we require $(Q, R)$ controllable for the signal generator \eqref{geneQR}. }

With the above assumptions, let \eqref{geneSL}, \eqref{geneQR} and $\bm{\Sigma}$ in \eqref{sys:2o} be interconnected, with $u = \theta$ and $\psi = y$, illustrated in Fig.~\ref{fig:interconnection}. 
\begin{figure}[h]\centering
\includegraphics[scale=.31]{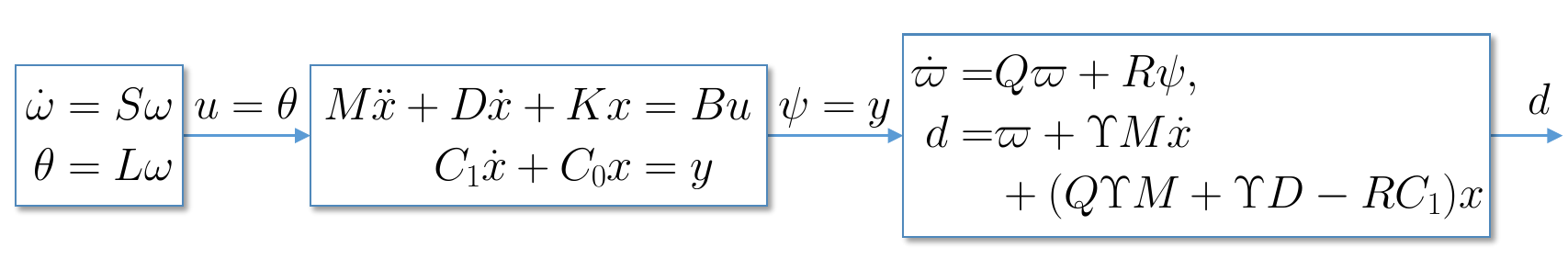}	
	%
\caption{Interconnection of $\bm{\Sigma}$ as in \eqref{sys:2o} with the signal generators \eqref{geneSL} and \eqref{geneQR}.}
\label{fig:interconnection}
\end{figure}

Following \cite{Astolfi2010MM,Ionescu2016TwoSided}, we show that the moments of system $\bm{\Sigma}$ at the interpolation points $\sigma(S)$ and $\sigma(Q)$
are characterized simultaneously by the steady-state response of signal $d(t)$. 
\begin{pro} \label{pro:interconnection}
{\rev Consider the signal generators \eqref{geneSL} and \eqref{geneQR}, where the triple $(L, S, \omega(0))$  is minimal, and the pair $(Q, R)$ is controllable.} Consider the two-sided interconnection of $\bm{\Sigma}$ with the signal generators, with $u = \theta$ and $\psi = y$. Then, on the manifold $\mathcal{M} = \{(x, \omega) \in \mathbb{R}^{n+\nu} \mid x = \Pi \omega\}$, it holds that
\begin{equation*}
	\dot{d} = Q d + \Upsilon BL \omega,
\end{equation*}
where $\Upsilon$ is the unique solution of equation \eqref{eq:Sylv2}.	
\end{pro} 
The proof of Proposition~\ref{pro:interconnection} can be found in Appendix~\ref{app:pro:interconnection}. 
With the above result, we are ready to determine the second-order reduced model of dimension $\nu$, that matches the moments of $\bm{\Sigma}$ at both $\sigma(S)$ and $\sigma(Q)$, respectively. Note that this model is within the families of second-order reduced models defined in \eqref{sys:2orG} and \eqref{sys:2orH} with a particular choice of $G$ and $H$, respectively.

\begin{thm} \label{thm:twoside}
Consider $\bm{\Sigma}$ as in \eqref{sys:2o} and let $S, Q \in \mathbb{R}^{\nu \times \nu}$ be such that $\sigma(S) \cap \Omega = \emptyset$ and $\sigma(Q) \cap \Omega = \emptyset$. Let
$L \in \mathbb{R}^{1 \times \nu}$, $R \in \mathbb{R}^{\nu}$ be such that the pair $(L, S)$ is observable and the pair $(Q,R)$ is controllable. Suppose $\Pi \in \mathbb{R}^{n \times \nu}$ and $\Upsilon \in \mathbb{R}^{\nu \times n}$ are the unique solutions of \eqref{eq:Sylv1} and \eqref{eq:Sylv2}, respectively, and $\Upsilon \Pi$ is nonsingular, and denote
\begin{equation} 
	\Pi^\dagger: = (\Upsilon \Pi)^{-1}\Upsilon, \ \text{and} \  \Upsilon^\dagger: = \Pi(\Upsilon \Pi)^{-1},
\end{equation}
the left pseudoinverse of $\Pi$ and the right pseudoinverse of $\Upsilon$, respectively. Let $\widehat{\Omega}$ be the set defined in \eqref{eq:hatOmega}, which satisfies $\sigma(S) \cap \widehat{\Omega}= \emptyset$ and $\sigma(Q) \cap \widehat{\Omega} = \emptyset$.
\begin{enumerate}
	\item The unique model $\bm{\widehat{\Sigma}}_G$ in \eqref{sys:2orG} that matches the moments of $\bm{\Sigma}$ at $\sigma(S)$ and $\sigma(Q)$ simultaneously is given by 
	\begin{equation} \label{eq:redMats}
		F_2  = \Pi^\dagger M \Pi, F_1 = \Pi^\dagger D \Pi,  
		G  = \Pi^\dagger B, H_1 = C_1 \Pi.
	\end{equation} 
	\item The unique model $\bm{\widehat{\Sigma}}_H$ in \eqref{sys:2orH} that matches the moments of $\bm{\Sigma}$ at $\sigma(S)$ and $\sigma(Q)$ simultaneously is given by
	\begin{equation}  \label{eq:redMats2}
		F_2 = \Upsilon M \Upsilon^\dagger, F_1 = \Upsilon D \Upsilon^\dagger,  
		H_1 = C_1 \Upsilon^\dagger, H_0 = C_0 \Upsilon^\dagger.
	\end{equation}
	\item The reduced models $\bm{\widehat{\Sigma}}_G$ and $\bm{\widehat{\Sigma}}_H$ are equivalent.
\end{enumerate}
\end{thm}
The proof of Theorem~\ref{thm:twoside} can be found in Appendix~\ref{app:thm:twoside}.

\subsection{Moment Matching With Pole-Zero Placement}
\label{sec:MM_polezero}
In this section, we extend the arguments in \cite{Datta2012poleplacement,Ionescu2021poleplacement} to consider the pole-zero placement problem in the reduced-order modeling of second-order systems.  

Specifically, we consider $\bm{\Sigma}$ in \eqref{sys:2o} and the family of approximations $\bm{\widehat{\Sigma}}_G$ as in \eqref{sys:2orG} that matches the moments of $\bm{\Sigma}$ at $\sigma(S)$ with $S \in \mathbb{R}^{\nu \times \nu}$. 

\paragraph*{Pole placement} For pole placement, the objective is to find the parameter matrices $F_1$, $F_2$, $G$, and $H_1$ such that $\bm{\widehat{\Sigma}}_G$ has the poles at prescribed locations $\lambda_1, \lambda_2, ..., \lambda_{m_\p}$, where $m_\p \leq \nu$, and $\lambda_i \notin \sigma(S) \cap \Omega$ with $\Omega$ defined in \eqref{Omega}.

Define $Q_\mathrm{p} \in \mathbb{R}^{m_\p \times m_\p}$ such that $\sigma(Q_{\p})  = \{\lambda_1, \lambda_2, ..., \lambda_{m_\p}\}$. Due to $\sigma(Q_{\p}) \cap \Omega = \emptyset$, the second-order Sylvester equation 
\begin{equation} \label{eq:Sylv_Qp}
Q_{\p}^2 \Upsilon_{\p} M   +   Q_{\p} \Upsilon_{\p} D + \Upsilon_{\p} K  = R_{\p} C_{{\p}0} + Q_{\p} R_{\p} C_{{\p}1}.
\end{equation}
has the unique solution $\Upsilon_{\p} \in \mathbb{R}^{m_\p \times n}$, where $R_{\p} \in \mathbb{R}^{m_\p}$ is any matrix such that the pair $(Q_{\p}, R_{\p})$ is controllable, and $C_{{\p}0}, C_{{\p}1} \in \mathbb{R}^{1 \times n}$ such that $C_{{\p}0} \Pi = C_{{\p}1} \Pi = 0$, i.e. $C_{{\p}0}^\top \in \ker (\Pi)$ and $C_{{\p}1}^\top \in \ker (\Pi)$ with $\Pi$ the unique solution of \eqref{eq:Sylv1}. 

Then, we impose linear constraints on the free parameters of the reduced model $\bm{\widehat{\Sigma}}_G$ such that the reduced
model $\bm{\widehat{\Sigma}}_G$ has poles at $\sigma(Q_\mathrm{p})$.
\begin{thm}\label{thm:pole}
Consider $\bm{\widehat{\Sigma}}_G$ in \eqref{sys:2orG} as a family of reduced models that
match the moments of the system \eqref{sys:2o} at $\sigma(S)$. Let $\Pi$ and
$\Upsilon_{\p} \in \mathbb{R}^{m_\p \times n}$ be the unique solutions of  \eqref{eq:Sylv1} and \eqref{eq:Sylv_Qp}, respectively. Assume that ${\rank}(\Upsilon_{\p} \Pi) = m_\p$. If the following constraints hold
\begin{subequations}\label{eq:constraints}
	\begin{align} 
		\Upsilon_{\p} \Pi F_2 & = \Upsilon_{\p} M \Pi,\\ \Upsilon_{\p} \Pi F_1 & = \Upsilon_{\p} D \Pi, \\ \Upsilon_{\p} \Pi G & = \Upsilon_{\p} B,
	\end{align}
\end{subequations}
then $\sigma(Q_{\p})  = \{\lambda_1, \lambda_2, ..., \lambda_{m_\p}\} \subseteq \widehat{\Omega}$ with $\widehat{\Omega}$ in \eqref{eq:hatOmega} the set of poles of the reduced model $\bm{\widehat{\Sigma}}_G$.
\end{thm}
The proof of Theorem~\ref{thm:pole} can be found in Appendix~\ref{app:thm:pole}. 

\begin{rem}
Theorem \ref{thm:pole} yields the sufficient conditions \eqref{eq:constraints} on the set $\G$ such that $m_\p\leq\nu$ of the poles of \eqref{sys:2orG} are fixed, when the pair $(L,S)$ is observable and the pair $(\Qp,\Rp)$ is controllable. Furthermore, if $m_\p=\nu$ and $\Up\Pi$ is assumed invertible, then $\what\Omega=\sigma(\Qp)$, if and only if 
\begin{align} \label{eq:redmod_pole}
	F_2  &= (\Up \Pi)^{-1}\Up M \Pi,\quad F_1 = (\Up \Pi)^{-1}\Up D \Pi, \nonumber \\ G  &= (\Up \Pi)^{-1}\Up B.
\end{align}
\end{rem}

\paragraph*{Zero placement}	{\rev Next, we discuss how to place zeros of the given second-order model in reduced-order models obtained through moment matching. First, the notion of \textit{zeros} for dynamical systems, as defined in \cite[Chapter 8]{astrom-murray-2008}, is extended to second-order systems. For a second-order system \eqref{sys:2o}, we determine the conditions such that for the input  $u(t)=u_0\e^{st}$ and the state evolution $x(t)=x_0\e^{st}$, with $u_0, x_0 \not =0$,  the resulting output satisfies $y(t)=0$, for all $t$. Substituting $u$ and $x$ in \eqref{sys:2o} yields
\begin{align*}
	\e^{st}(s^2Mx_0+sDx_0+Kx_0-Fu_0)&=0, \\
	\e^{st}(C_1s+C_0)&=0.
\end{align*}
%
Hence, $s=z$ is a \textit{zero} of the system in \eqref{sys:2o} if 
\begin{align}\label{eq_zeros_2o}
	\det\begin{bmatrix}
		z^2M+zD+K & -F\\ C_1z+C_0 & 0
	\end{bmatrix} =0.
\end{align}

For zero placement, the objective is to find the parameter matrices $F_1$, $F_2$, $G$, and $H_1$ such that $\bm{\widehat{\Sigma}}_G$ has zeros at prescribed locations $z_1, z_2, ..., z_{m_\z}$, where $m_\z < \nu$, and $z_i \notin \sigma(S) \cap \Omega$ with $\Omega$ defined in \eqref{Omega}. Applying the definition \eqref{eq_zeros_2o} to the family of reduced-order models $\bm{\widehat{\Sigma}}_G$ defined in \eqref{sys:2orG}, we obtain that $z_1, ..., z_{m_\z}$ are zeros of the system \eqref{sys:2orG} if the following equation holds for all $i = 1, 2, ..., m_\z$:
\begin{equation}\label{eq_assign_zeros}
	\det\!\left[\!\!\begin{array}{cc} F_2z_i^2+F_1z_i+(GL-F_2S^2-F_1S) &  -G \\ H_1z_i+(C_0\Pi+C_1\Pi S-H_1S) & 0\end{array}\!\!\right]\!=\!0,\end{equation}
%

Now let $\Qz\in\mathbb R^{m_\z\times m_\z}$ with $\sigma(\Qz)=\{z_1,\dots,z_{m_\z}\}$ and $\Rz\in\mathbb R^{m_\z}$ be any matrix such that the pair $(\Qz,\Rz)$ is controllable. Let $\Uz\in \mathbb R^{m_\z\times n}$ be the unique solution of the Sylvester equation
\begin{equation}\label{eq:Sylv_Uz}
	\Qz^2 \Uz M   +   \Qz\Uz D + \Uz K  = \Rz C_0 + \Qz \Rz C_1.
\end{equation}
with $\rank\Uz=m_\z$. The moments of $W(s)$ at $z_i$ are given by $\Uz B$. Assuming  $W(z_i)=0$, then $\Uz B=0$. The next result imposes linear constraints on the reduced-order models $\G$ such that they have $m_\z$ zeros at $\{z_1,\dots,z_{m_\z}\}$.
\begin{thm}\label{thm:zero}
	Consider $\G$, as in \eqref{sys:2orG}, a family of reduced systems of order $\nu$
	matching the moments of \eqref{sys:2o} at $\sigma(S)$. Consider the matrix $\Qz\in\mathbb R^{m_\z\times m_\z}$ with $\sigma(\Qz)=\{z_1,\dots, z_{m_\z}\}$, a symmetric set and let $\Rz \in \mathbb{R}^{m_\z}$ be such that the pair $(\Qz,\Rz)$ is controllable. Let $\Pi$ and
	$\Uz \in \mathbb{R}^{m_\z \times n}$ be the unique solutions of \eqref{eq:Sylv1} and \eqref{eq:Sylv_Uz}, respectively. Assume that $\mathrm{rank}(\Uz \Pi) = \ell$. If the following constraints hold
	\begin{subequations}\label{eq:constraints_z2}
		\begin{align} 
			H_1&=C_1\Pi,\label{eq_zeros_assign_H1} \\ \Uz \Pi F_2 & = -\Uz M \Pi,\label{eq_zeros_assign_F2}\\ \Uz \Pi F_1 & = -\Uz D \Pi, \label{eq_zeros_assign_F1}\\ \Uz \Pi G & = 0, \label{eq_zeros_assign_G}
		\end{align}
	\end{subequations}
	then $\{z_1,\dots, z_{m_\z}\}=\sigma(\Qz)$ are zeros of the system $\G$.
\end{thm}
The proof can be found in Appendix~\ref{app:thm:zero}.}

\paragraph*{Pole-zero placement} Let $\what{\bm\Sigma}_G$, as in \eqref{sys:2orG}, define a family of $\nu$ order models that match $\nu$ moments  of \eqref{sys:2o} at $\{s_1,\dots,s_{\nu}\},$ parameterized in the set of matrices $\{F_1,F_2,G,H_1\}$ of appropriate dimensions. Let $\{\lambda_1,\dots,\lambda_{m_\p}\}$ and $\{z_1,\dots,z_{m_\z}\}$ be symmetric sets (including multiplicities), such that $\{\lambda_1,\dots,\lambda_{m_\p}\}\cap\Omega=\emptyset$, $\{\lambda_1,\dots,\lambda_{m_\p}\}\cap\sigma(S)=\emptyset$ and $\{z_1,\dots,z_{m_\z}\}\cap\sigma(S)=\emptyset$, ${m_\p}+{m_\z}\leq\nu$. 
We now collect the constraints \eqref{eq:constraints} and \eqref{eq:constraints_z2}, yielding the system of matrix equations in the unknowns $F_1,F_2,G,H_1$.
\begin{cor}\label{cor:allconstraints2}
Let $\what{\bm\Sigma}_G$, as in \eqref{sys:2orG}, define a family of $\nu$ order models that match $\nu$ moments  of \eqref{sys:2o} at $\{s_1,\dots,s_{\nu}\}$. Let $\Pi$ be the solution of the matrix equation \eqref{eq:Sylv1}, $\Up$ be the solution of the matrix equation \eqref{eq:Sylv_Qp}, and $\Uz$ be the solution of the matrix equation \eqref{eq:Sylv_Uz}. Denote $\Upsilon=\begin{bmatrix}\Up^\top &\Uz^\top \end{bmatrix}^\top\in\mathbb R^{\nu\times n}$. 
Assuming \eqref{eq_zeros_assign_H1} holds and if 
\begin{subequations}\label{eq:allconstraints2}
	\begin{align}
		\Upsilon \Pi F_2 &=\begin{bmatrix}
			(\Up M\Pi)^\top & -(\Uz M\Pi)^\top
		\end{bmatrix}^\top, \\
		\Upsilon \Pi F_1 &=\begin{bmatrix}
			(\Up D\Pi)^\top & -(\Uz D\Pi)^\top
		\end{bmatrix}^\top, \\
		\Upsilon \Pi G &=\begin{bmatrix}
			(\Up B)^\top & 0
		\end{bmatrix}^\top,
	\end{align}
\end{subequations}
then $\{\lambda_1,\dots,\lambda_{m_\p}\}$ are poles and  $\{z_1,\dots,z_{m_\z}\}$ are zeros of $\widehat{\bm\Sigma}_G$ as in \eqref{sys:2orG}, respectively.
\end{cor}
\begin{proof}
The proof follows using arguments from Theorem~\ref{thm:pole} and Theorem~\ref{thm:zero}. Hence it is omitted.
\end{proof}



Note that if ${m_\p}+{m_\z}<\nu$, then the sufficient conditions expressed through the linear systems \eqref{eq:allconstraints2} have an infinite number of solutions, respectively. Then, there exist solutions $F_2, F_1, G \in \G$ such that additional constraints can be imposed (e.g., diagonal mass matrix $F_2$ and/or stiffness matrix $F_0$ symmetric). 
Furthermore, if ${m_\p}+{m_\z}=\nu$ and $\Upsilon\Pi$ is invertible, then 
\begin{subequations}\label{eq:allconstraints2_lknu}
\begin{align}
	F_2 & \!=\!( \Upsilon \Pi)^{-1}\!\!\begin{bmatrix}
		\Up M\Pi \\ -\Uz M\Pi
	\end{bmatrix}, \
	F_1\! =\!( \Upsilon \Pi)^{-1}\!\!\begin{bmatrix}
		\Up D\Pi \\ -\Uz D\Pi
	\end{bmatrix}, \label{eq:allconstraints2_lknuF}\\
	G &= ( \Upsilon \Pi)^{-1}\begin{bmatrix}
		\Up B \\ 0 \label{eq:allconstraints2_lknuG}
	\end{bmatrix},
\end{align}
\end{subequations}

provides the unique model \eqref{sys:2orG} having the poles $\lambda_i, i=1:{m_\p}$ and the zeros $z_j, j=1:{m_\z}$. 

The constructive steps for computing a reduced model satisfying the pole-zero constraints are summarized in Algorithm~\ref{alg_PZD_G}.
\begin{algorithm}[t]
\caption{Second-order time-domain moment matching with prescribed poles and zeros}\label{alg_PZD_G}
\begin{algorithmic}
	\STATE 
	\STATE {\textsc{GIVEN DATA:}}
	\STATE \hspace{0.5cm} \textbf{A second-order system \eqref{sys:2o}};
	\STATE \hspace{0.5cm} $\nu < n\in\mathbb N$, chosen;
	\STATE \hspace{0.5cm} interpolation points $\{s_i\in\mathbb C\setminus \Omega \ |\ i=1:\nu\}$; 
	\STATE \hspace{0.5cm} poles $\{\lambda_1,\dots,\lambda_{m_\p}\}\subset\mathbb{C},\ s_j\neq\lambda_j, j=1:m_\p$;
	\STATE \hspace{0.5cm} zeros $\{z_1,\dots,z_{m_\z}\}\subset\mathbb{C}$, $m_\p+m_\z\leq\nu$;
	\STATE {\textsc{COMPUTE:}}
	\STATE \hspace{0.5cm} $S\in\mathbb R^{\nu\times\nu}$, such that $\sigma(S)=\{s_i\ |\ i=1:\nu\}$;
	\STATE \hspace{0.5cm} $L\in\mathbb R^{1\times\nu}$ such that $(L,S)$ is observable;
	\STATE \hspace{0.5cm} $\Pi\in\mathbb R^{n\times\nu}$, the solution of \eqref{eq:Sylv1};
	\STATE \hspace{0.5cm} $\Up$ and $\Uz$, the solutions of \eqref{eq:Sylv_Qp} and \eqref{eq:Sylv_Uz};
	\STATE {\textsc{SOLUTION:}}
	\STATE \hspace{0.5cm} $F_1, F_2, G$, as in \eqref{eq:allconstraints2_lknu};
	\STATE \hspace{0.5cm} Substitute $F_1, F_2, G$ and $H_1=C_1\Pi$, into $\what{\bm\Sigma}_G$ in \eqref{sys:2orG}.
\end{algorithmic}
\label{alg1}
\end{algorithm}
\subsection{Moment Matching of First-Order Derivatives}
\label{sec:firstorder}


\rev In the context of moment matching-based model reduction, matching the first-order derivatives of the transfer functions at specified interpolation points is also an important question. From the theoretical perspective, matching first-order derivatives at the same interpolation points is necessary for the first-order necessary optimality conditions associated with the minimal ${H}_2$ norm error approximation problem \cite{gugercin-antoulas-beattie-SIAM2008}. Furthermore, it has often been observed from numerical simulations that matching the first-order derivatives at the interpolation points leads to a notably smaller reduction error in the  ${H}_2$ norm. Therefore, in this section, we study the moment matching of first derivatives of second-order transfer functions. Specifically, we focus on the reduced second-order systems   that match the moments of both zero and first-order derivatives of the transfer function
$ 
W(s) = C (Ms^2 + Ds + K)^{-1}B,
$ 
where $C = C_0$ and $C_1 = 0$ in the original system \eqref{sys:2o}. 

Denote 
\begin{subequations} \label{eq:WLWR}
\begin{align}
	W_L(s): &= -C(Ms^2 + Ds + K)^{-1}, \label{eq:WL}\\
	W_R(s): &= (2Ms +D)(Ms^2 + Ds + K)^{-1}B.
	\label{eq:WR}
\end{align}
\end{subequations} 
Then, the first-order derivative of $W(s)$ is 
$
W^\prime (s) =W_L(s) \cdot W_R(s),
$
which has a state-space representation as
\begin{equation} \label{sysdif}
\bm{\Sigma}^\prime: \
\begin{cases}
	M \ddot{x}(t) + D \dot{x}(t) + K x(t) & = B u(t), \\
	M \ddot{z}(t) + D \dot{z}(t) + K z(t) &= 2M \dot{x}(t) + D x(t), \\
	- Cz(t) & =  y(t),
\end{cases}
\end{equation} 
with $z(t) \in \mathbb{R}^n$ and $y(t) \in \mathbb{R}$. 

Consider the following signal generator   
\begin{align}\label{geneSL1}
&\dot{\varpi} = S \varpi + L^\top \psi, \varpi(0) = 0, \nonumber \\ 
&d = \varpi + (S\Upsilon M + \Upsilon D) z + \Upsilon M \dot{z},
\end{align}
where $\Upsilon \in \mathbb{R}^{\nu \times n}$ is the unique solution of the second-order Sylvester equation:
\begin{equation} \label{eq:Sylv12}
S^2 \Upsilon M + S \Upsilon D + \Upsilon K = - L^\top C,
\end{equation}
since $\sigma(S) \cap \Omega = \emptyset$ is assumed.
We then connect the system $\bm{\Sigma}^\prime$ with the signal generators \eqref{geneSL} and \eqref{geneSL1}, where $u = \theta$ and $\psi = y$, see Fig.~\ref{fig:interconnection1}. The following result is obtained with the property of the signal $d(t)$ in \eqref{geneSL1}.
\begin{figure}[h]\centering
\includegraphics[scale=.31]{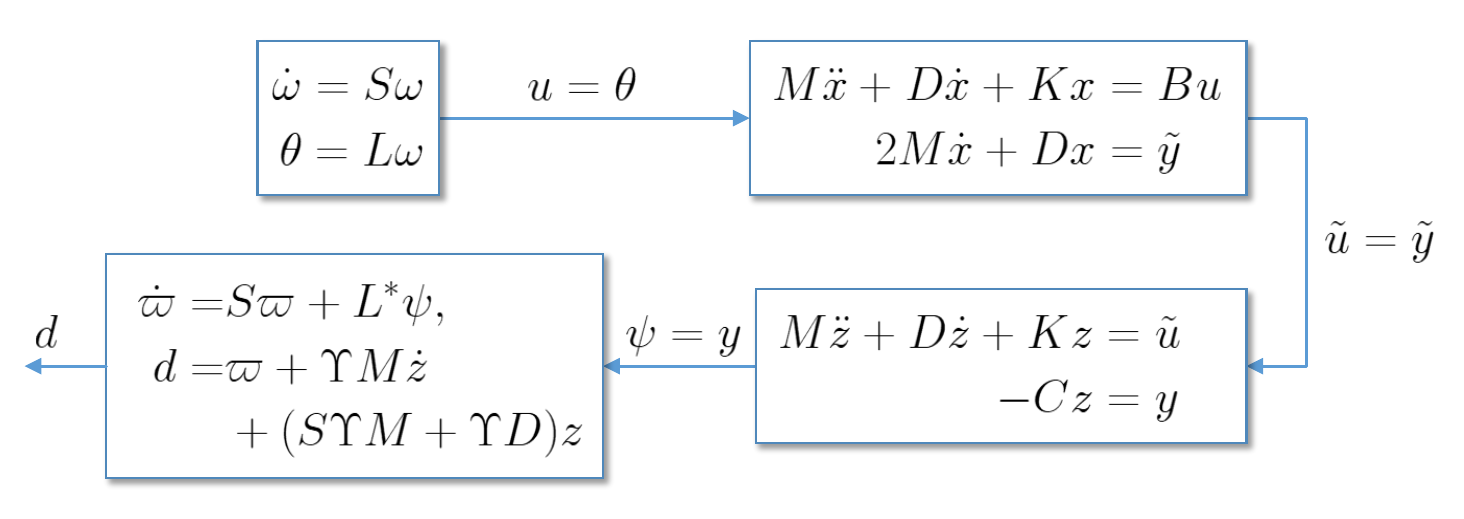}	
\caption{Illustration of the interconnection of $\bm{\Sigma}^\prime$ with the signal generators \eqref{geneSL} and \eqref{geneSL1}.}
\label{fig:interconnection1}
\end{figure}
\begin{thm} \label{thm:one2one}
Consider the system $\bm{\Sigma}^\prime$ in \eqref{sysdif}, which is connected to the signal generators \eqref{geneSL} and \eqref{geneSL1} with $u = \theta$ and $\psi = y$. Let $\Pi$ and
$\Upsilon$ be the unique solutions of \eqref{eq:Sylv1} and \eqref{eq:Sylv12}, respectively, and assume that $\Omega \subset \mathbb{C}^-$ and $\sigma(S) \subset \mathbb{C}^0$. Then the moments of $\bm{\Sigma}^\prime$ at $\sigma(S)$  
are in a one-to-one relation with the steady-state response of the signal $d(t)$ in \eqref{geneSL1}.
\end{thm}
The proof of Theorem~\ref{thm:one2one} is shown in Appendix~\ref{app:thm:one2one}. 

Next, we present a second-order reduced model that matches the moments of $W(s)$ and $W^\prime(s)$ simultaneously at the interpolation points $\sigma(S)$. We therefore suppose $H_1 = 0$ and $H_0 = H$ in \eqref{sys:2orG} and  \eqref{sys:2orH}.
\begin{thm} \label{thm:equivalence}
Consider a linear second-order system $\bm{\Sigma}$ in \eqref{sys:2o} and let $S  \in \mathbb{R}^{\nu \times \nu}$, $L  \in \mathbb{R}^{1 \times \nu}$ be such that the pair $(L, S)$ is observable, and $\Pi \in \mathbb{R}^{n \times \nu}$ and $\Upsilon \in \mathbb{R}^{\nu \times n}$ are the unique solutions of \eqref{eq:Sylv1} and \eqref{eq:Sylv12}, respectively, such that $\Upsilon \Pi$ is nonsingular. Then the following statements hold.
\begin{enumerate}
	\item A model $\bm{\widehat \Sigma}_G$  that matches the moments of $W(s)$ and $W^\prime(s)$ 
	at $\sigma(S)$ is given by 
	\begin{equation} \label{eq:RedMats}
		F_2  = \Pi^\dagger M \Pi, F_1 = \Pi^\dagger D \Pi,  
		G   = \Pi^\dagger B,
	\end{equation} 
	with $
	\Pi^\dagger: = (\Upsilon \Pi)^{-1}\Upsilon.
	$
	
	\item A model $\bm{\widehat{\Sigma}}_H$  that matches  the moments of $W(s)$ and $W^\prime(s)$ 
	at $\sigma(S)$ is given by
	\begin{equation*}  
		F_2 = \Upsilon M \Upsilon^\dagger, F_1 = \Upsilon D \Upsilon^\dagger,  
		H = C \Upsilon^\dagger,
	\end{equation*}
	with $\Upsilon^\dagger: = \Pi(\Upsilon \Pi)^{-1}
	$.
	
	\item The reduced models $\bm{\widehat{\Sigma}}_G$ and $\bm{\widehat{\Sigma}}_H$ are equivalent.	
\end{enumerate}
\end{thm}
The proof of Theorem~\ref{thm:equivalence} is shown in Appendix~\ref{app:thm:equivalence}. 

\section{Towards Moment Matching for MIMO Second-Order Systems}
\label{sec:mimo_extension}

The previous sections develop a complete second-order time-domain moment matching framework for SISO systems, including two-sided matching, pole-zero placement, and first-order derivative matching. For MIMO systems, the basic second-order Sylvester mechanism still carries over, but the interpolation data must be formulated tangentially. In this section, we provide the MIMO results that follow naturally from the SISO framework and identify the additional difficulties that prevent a direct extension of the derivative-matching and pole-zero placement results.

Consider the second-order system \eqref{sys:2o} with $p$ inputs and $q$ outputs, i.e.,
$
W(s)=(C_1s+C_0)(Ms^2+Ds+K)^{-1}B\in\mathbb{C}^{q\times p}.
$
In contrast with the SISO case, an interpolation point alone does not determine a scalar moment. One must also specify input and output directions.

\begin{defn}[Tangential MIMO moments]
\label{def:mimo_tangential_moments}
Let $s_\star\in\mathbb{C}\setminus\Omega$. For a right direction $\ell\in\mathbb{C}^{p}$ and a left direction $\rho\in\mathbb{C}^{q}$, the $k$th right and left tangential moments are defined as
\begin{align*}
	\eta_k^{\rm r}(s_\star,\ell)
	&=\frac{(-1)^k}{k!}
	\left[\frac{\opd^k W(s)}{\opd s^k}\right]_{s=s_\star}\ell,\\
	\eta_k^{\rm l}(s_\star,\rho)
	&=\frac{(-1)^k}{k!}
	\rho^\ast\left[\frac{\opd^k W(s)}{\opd s^k}\right]_{s=s_\star}.
\end{align*}
\end{defn}

Consider the following interpolation matrices
\[
S,Q\in\mathbb{C}^{\nu\times\nu},\qquad
L\in\mathbb{C}^{p\times\nu},\qquad
R\in\mathbb{C}^{\nu\times q},
\]
where $(L,S)$ is observable and $(Q,R)$ is controllable. The right and left second-order Sylvester equations are
\begin{subequations}\label{eq:mimo_sylvester}
\begin{align}
	M\Pi S^2+D\Pi S+K\Pi&=BL, \label{eq:mimo_right_sylv}\\
	Q^2\Upsilon M+Q\Upsilon D+\Upsilon K&=RC_0+QRC_1.
	\label{eq:mimo_left_sylv}
\end{align}
\end{subequations}

\begin{lem}
\label{lem:mimo_tangential_identities}
Assume that $S$ and $Q$ are diagonalizable and that
$\sigma(S)\cap\Omega=\emptyset$ and $\sigma(Q)\cap\Omega=\emptyset$.
Let $\Pi$ and $\Upsilon$ solve \eqref{eq:mimo_right_sylv} and
\eqref{eq:mimo_left_sylv}, respectively. If $Sv_i=s_iv_i$ and
$w_j^\ast Q=\mu_jw_j^\ast$, define
$
\ell_i:=Lv_i$, and $\rho_j^\ast:=w_j^\ast R.
$
Then
\begin{align}
	W(s_i)\ell_i&=(C_0\Pi+C_1\Pi S)v_i,
	\label{eq:mimo_right_tangential_identity}\\
	\rho_j^\ast W(\mu_j)&=w_j^\ast\Upsilon B.
	\label{eq:mimo_left_tangential_identity}
\end{align}
\end{lem}
\begin{proof}
Let $P(s):=Ms^2+Ds+K$. Multiplying \eqref{eq:mimo_right_sylv} by $v_i$ gives
$
P(s_i)\Pi v_i=B\ell_i.
$
Since $s_i\notin\Omega$, the matrix $P(s_i)$ is nonsingular, and therefore
$
\Pi v_i=P(s_i)^{-1}B\ell_i.
$
Premultiplying by $C_0+s_iC_1$ yields \eqref{eq:mimo_right_tangential_identity}. Similarly, premultiplying \eqref{eq:mimo_left_sylv} by $w_j^\ast$ gives
$
w_j^\ast\Upsilon P(\mu_j)=\rho_j^\ast(C_0+\mu_jC_1).
$
Since $\mu_j\notin\Omega$, right multiplication by $P(\mu_j)^{-1}B$ yields \eqref{eq:mimo_left_tangential_identity}.
\end{proof}

\begin{rem}[Real realizations for conjugate data]
\label{rem:mimo_realification}
If the complex tangential data are closed under conjugation, then the equations above can be written with real matrices. Indeed, each pair $a\pm\mathrm{i}b$ can be represented by the real block
$
\begin{bmatrix}
	a & b\\
	-b & a
\end{bmatrix},
$
and the corresponding direction columns are transformed by the same similarity transformation. The Sylvester equations are invariant under this change of basis. Consequently, the transformed matrices $S,L,Q,R$, the solutions $\Pi,\Upsilon$, and the reduced models constructed below can all be chosen real while preserving the same tangential interpolation conditions.
\end{rem}
Following a similar reasoning in the proof of Theorem~\ref{thm:Moments}, we can derive the following result.
\begin{thm}
\label{thm:mimo_moments}
Assume that $M$ is nonsingular. If
$\sigma(S)\cap\Omega=\emptyset$ and $\sigma(Q)\cap\Omega=\emptyset$, then
\eqref{eq:mimo_right_sylv} and \eqref{eq:mimo_left_sylv} have unique
solutions. Moreover, these solutions characterize the right and left
zeroth-order tangential moments through
\eqref{eq:mimo_right_tangential_identity} and
\eqref{eq:mimo_left_tangential_identity}.
\end{thm}


\subsection{Tangential Moment Matching and Second-Order Approximation}
\label{subsec:mimo_tangential_matching}

We now use the tangential moment characterization above to construct
second-order reduced models for MIMO systems. The main idea is similar to
the SISO case: the Sylvester solutions associated with the original system
encode the target moments, while the corresponding Sylvester solutions
associated with the reduced system encode the moments of the approximation.
Moment matching is therefore obtained by equating these two encoded
quantities.

Consider a MIMO reduced second-order model
\begin{equation}\label{sys:mimo_reduced}
\bm{\widehat\Sigma}:\quad
\begin{cases}
	F_2\ddot\xi+F_1\dot\xi+F_0\xi=Gu,\\
	H_1\dot\xi+H_0\xi=\psi,
\end{cases}
\end{equation}
where $\xi\in\mathbb{R}^{\nu}$, $u\in\mathbb{R}^{p}$, $\psi\in\mathbb{R}^{q}$,
$F_i\in\mathbb{R}^{\nu\times\nu}$, $G\in\mathbb{R}^{\nu\times p}$, and
$H_0,H_1\in\mathbb{R}^{q\times\nu}$. Its transfer matrix is
\[
\widehat W(s)=(H_1s+H_0)(F_2s^2+F_1s+F_0)^{-1}G.
\]
Let
$
\widehat\Omega:=\{s\in\mathbb{C}:\det(F_2s^2+F_1s+F_0)=0\}.
$
The following result gives the
necessary and sufficient algebraic conditions under which
\eqref{sys:mimo_reduced} matches the prescribed right or left tangential
moments.
\begin{pro}
\label{pro:mimo_one_sided}
Assume that the reduced pencil is regular.
\begin{enumerate}
	\item If $\sigma(S)\cap\widehat\Omega=\emptyset$, then
	$\bm{\widehat\Sigma}$ matches the right tangential moments encoded by
	$(S,L)$ if and only if
	\begin{align}
		C_0\Pi+C_1\Pi S&=H_0P_{\rm r}+H_1P_{\rm r}S,
		\label{eq:mimo_right_output_match}\\
		F_2P_{\rm r}S^2+F_1P_{\rm r}S+F_0P_{\rm r}&=GL,
		\label{eq:mimo_right_reduced_sylv}
	\end{align}
	where $P_{\rm r}$ is the unique solution of
	\eqref{eq:mimo_right_reduced_sylv}.
	
	\item If $\sigma(Q)\cap\widehat\Omega=\emptyset$, then
	$\bm{\widehat\Sigma}$ matches the left tangential moments encoded by
	$(Q,R)$ if and only if
	\begin{align}
		\Upsilon B&=P_{\rm l}G,
		\label{eq:mimo_left_input_match}\\
		Q^2P_{\rm l}F_2+QP_{\rm l}F_1+P_{\rm l}F_0&=RH_0+QRH_1,
		\label{eq:mimo_left_reduced_sylv}
	\end{align}
	where $P_{\rm l}$ is the unique solution of
	\eqref{eq:mimo_left_reduced_sylv}.
\end{enumerate}
\end{pro}

\begin{proof}
The proof follows by applying Theorem~\ref{thm:mimo_moments} first to the
original system and then to the reduced system. For the right tangential
data, the matrix $C_0\Pi+C_1\Pi S$ encodes the moments of $W(s)$ associated
with $(S,L)$. Since $\sigma(S)\cap\widehat\Omega=\emptyset$, the reduced
Sylvester equation \eqref{eq:mimo_right_reduced_sylv} has a unique solution
$P_{\rm r}$, and the matrix $H_0P_{\rm r}+H_1P_{\rm r}S$ encodes the
corresponding right tangential moments of $\widehat W(s)$. Hence equality
in \eqref{eq:mimo_right_output_match} is equivalent to right tangential
moment matching. The left statement is obtained in the same way from the
left Sylvester equation \eqref{eq:mimo_left_reduced_sylv}: the matrices
$\Upsilon B$ and $P_{\rm l}G$ encode, respectively, the left tangential
moments of $W(s)$ and $\widehat W(s)$.
\end{proof}

A convenient right-interpolating family is obtained by fixing the reduced Sylvester
solution to be $P_{\rm r}=I_\nu$. Suppose $P_{\rm r}=I_\nu$, we obtain the right-interpolating family
\begin{equation}\label{sys:mimo_right_family}
\bm{\widehat\Sigma}_{G}^{\rm MIMO}: \
\begin{cases}
	F_2\ddot\xi+F_1\dot\xi+(GL-F_2S^2-F_1S)\xi=Gu,\\
	H_1\dot\xi+(C_0\Pi+C_1\Pi S-H_1S)\xi=\psi,
\end{cases}
\end{equation}
parameterized by $F_1,F_2,G,H_1$. Choosing $P_{\rm l}=I_\nu$ then gives the
left-interpolating family
{\small \begin{equation}\label{sys:mimo_left_family}
	\bm{\widehat\Sigma}_{H}^{\rm MIMO}: \
	\begin{cases}
		F_2\ddot\xi+F_1\dot\xi+(RH_0+QRH_1-Q^2F_2-QF_1)\xi
		\\
		\hspace*{5.8cm}=\Upsilon Bu,
		\\
		H_1\dot\xi+H_0\xi=\psi,
	\end{cases}
	\end{equation}}
	parameterized by $F_1,F_2,H_0,H_1$.
	
	\subsection{Two-Sided Tangential Matching}
	\label{subsec:mimo_two_sided}
	
	The one-sided families in \eqref{sys:mimo_right_family} and
	\eqref{sys:mimo_left_family} match either the right or the left tangential
	data. To match both sets of data with a single reduced model, the right
	interpolation subspace generated by $\Pi$ and the left interpolation subspace
	generated by $\Upsilon$ must be compatible. This compatibility is expressed
	by the nonsingularity of
	\[
	T:=\Upsilon\Pi .
	\]
	When this condition holds, the two Sylvester solutions define a
	second-order Petrov-Galerkin projection.
	Define $(A_2,A_1,A_0):=(M,D,K)$ and set
	\begin{align}
F_i^{\rm r}&=T^{-1}\Upsilon A_i\Pi,\quad i=0,1,2,
\nonumber\\
G^{\rm r}&=T^{-1}\Upsilon B, \quad
H_j^{\rm r}=C_j\Pi,\quad j=0,1.
\label{eq:mimo_projected_right}
\end{align}
The superscript ``${\rm r}$'' indicates the realization in which the right
interpolation conditions are enforced directly through the choice of the
reduced coordinates. 	Premultiplying \eqref{eq:mimo_right_sylv} by $T^{-1}\Upsilon$ gives
\[
F_2^{\rm r}S^2+F_1^{\rm r}S+F_0^{\rm r}=G^{\rm r}L.
\]
Moreover, $H_0^{\rm r}+H_1^{\rm r}S=C_0\Pi+C_1\Pi S$. Thus, the model belongs
to the right-interpolating family \eqref{sys:mimo_right_family}.  Next,
multiplying \eqref{eq:mimo_left_sylv} by $\Pi$ from the right gives
\[
Q^2\Upsilon M\Pi+Q\Upsilon D\Pi+\Upsilon K\Pi
=RC_0\Pi+QRC_1\Pi.
\]
Using $T=\Upsilon\Pi$, this is precisely
\eqref{eq:mimo_left_reduced_sylv} with $P_{\rm l}=T$ for the realization
\eqref{eq:mimo_projected_right}, while
\eqref{eq:mimo_left_input_match} follows from $TG^{\rm r}=\Upsilon B$.
Therefore the same model also matches the left tangential data. 

Similarly, we can show that the following
realization  
\begin{align}
F_i^{\rm l}&=\Upsilon A_i\Pi T^{-1},\quad i=0,1,2,
\nonumber\\
G^{\rm l}&=\Upsilon B, \quad
H_j^{\rm l} =C_j\Pi T^{-1},\quad j=0,1.
\label{eq:mimo_projected_left}
\end{align}
matches the left tangential moments encoded by $(Q,R)$. Furthermore, the two realizations are related by the second-order coordinate
transformation, i.e.,
\[
F_i^{\rm l}=TF_i^{\rm r}T^{-1},\quad
G^{\rm l}=TG^{\rm r},\quad
H_j^{\rm l}=H_j^{\rm r}T^{-1}.
\]

Next, we give the algebraic construction of a two-sided MIMO
second-order reduced model. The same construction also has a time-domain
interpretation analogous to the SISO case: the two signal generators encode
the right and left interpolation data, and the auxiliary signal $d$ carries the coupled two-sided information. 
Consider the signal generators
\begin{align}
\dot\omega&=S\omega, & \theta&=L\omega,\label{eq:mimo_geneSL}\\
\dot\varpi&=Q\varpi+R\psi, &
d&=\varpi+(Q\Upsilon M+\Upsilon D-RC_1)x+\Upsilon M\dot x.
\label{eq:mimo_geneQR}
\end{align}
Here $\Upsilon$ is computed beforehand from \eqref{eq:mimo_left_sylv}; hence
\eqref{eq:mimo_geneQR} defines an auxiliary output and does not introduce an
implicit algebraic constraint.

\begin{pro}
\label{pro:mimo_interconnection}
Interconnect \eqref{sys:2o}, \eqref{eq:mimo_geneSL}, and
\eqref{eq:mimo_geneQR} through $u=\theta$ and $\psi=y$. On the invariant
manifold $x=\Pi\omega$, one has
\[
\dot d=Qd+\Upsilon BL\omega.
\]
\end{pro}

\begin{proof}
Differentiate $d$ and substitute the plant dynamics and the generator
equations. This gives
\[
\dot d
=Q\varpi+Ry+(Q\Upsilon M+\Upsilon D-RC_1)\dot x
+\Upsilon M\ddot x.
\]
Using $y=C_1\dot x+C_0x$ and
$M\ddot x=BL\omega-D\dot x-Kx$, the terms containing $C_1\dot x$ and
$D\dot x$ cancel, yielding
\[
\dot d=Q\varpi+Q\Upsilon M\dot x+(RC_0-\Upsilon K)x+\Upsilon BL\omega.
\]
By \eqref{eq:mimo_left_sylv},
$RC_0-\Upsilon K=Q(Q\Upsilon M+\Upsilon D-RC_1)$. Therefore
\begin{align*}
	\dot d
	=Q\!\left[\varpi+(Q\Upsilon M+\Upsilon D-RC_1)x+\Upsilon M\dot x\right]
	+\Upsilon BL\omega
	\\=Qd+\Upsilon BL\omega.
\end{align*}
This completes the proof.
\end{proof}

A useful consequence of the two-sided construction is that, at interpolation
points shared by the right and left data, the reduced model also satisfies a
bitangential first-derivative interpolation condition. This is weaker than
full MIMO derivative matching, but it shows precisely which Hermite condition is inherited from two-sided tangential matching.

\begin{thm}[Bitangential Hermite interpolation]
\label{thm:mimo_hermite}
Let $\widehat W$ be the transfer matrix of the projected reduced model
\eqref{eq:mimo_projected_right}. Suppose that a simple point $s_i$ belongs to
both $\sigma(S)$ and $\sigma(Q)$, with
$Sv_i=s_iv_i$ and $w_i^\ast Q=s_iw_i^\ast$. Let
$\ell_i=Lv_i$ and $\rho_i^\ast=w_i^\ast R$. If $s_i$ is not a pole of the
reduced model, then
\begin{align}
	\widehat W(s_i)\ell_i&=W(s_i)\ell_i, &
	\rho_i^\ast\widehat W(s_i)&=\rho_i^\ast W(s_i),
	\label{eq:mimo_two_sided_values}\\
	\rho_i^\ast\widehat W'(s_i)\ell_i&=\rho_i^\ast W'(s_i)\ell_i.
	\label{eq:mimo_bitangential_hermite}
\end{align}
\end{thm}

\begin{proof}
Let $V:=\Pi$ and $W_{\rm p}^\ast:=T^{-1}\Upsilon$. Then
$W_{\rm p}^\ast V=I_\nu$ and the projected matrices satisfy
\[
\widehat P(s)=W_{\rm p}^\ast P(s)V,\quad
\widehat N(s)=(C_1s+C_0)V,
\]
where $\widehat P(s)=F_2^{\rm r}s^2+F_1^{\rm r}s+F_0^{\rm r}$ and
$\widehat N(s)=H_1^{\rm r}s+H_0^{\rm r}$. From the right Sylvester equation,
$
P(s_i)Vv_i=B\ell_i,
$
and hence $\widehat P(s_i)v_i=G^{\rm r}\ell_i$. From the left Sylvester
equation,
$
w_i^\ast\Upsilon P(s_i)=\rho_i^\ast(C_1s_i+C_0).
$
These two identities imply the two value-matching conditions in
\eqref{eq:mimo_two_sided_values}. For the derivative, use
\[
W'(s)=C_1P(s)^{-1}B-(C_1s+C_0)P(s)^{-1}P'(s)P(s)^{-1}B,
\]
with $P'(s)=2sM+D$. The right identity gives
$P(s_i)^{-1}B\ell_i=Vv_i$, and the left identity gives
$\rho_i^\ast(C_1s_i+C_0)P(s_i)^{-1}=w_i^\ast\Upsilon$. Therefore
\[
\rho_i^\ast W'(s_i)\ell_i
=\rho_i^\ast C_1Vv_i-w_i^\ast\Upsilon P'(s_i)Vv_i.
\]
The reduced transfer matrix gives the same expression, because
$\widehat P'(s_i)=W_{\rm p}^\ast P'(s_i)V$ and
$w_i^\ast T W_{\rm p}^\ast=w_i^\ast\Upsilon$. This proves
\eqref{eq:mimo_bitangential_hermite}.
\end{proof}

\subsection{Difficulties in Extending the Derivative and Pole-Zero Results to MIMO}
\label{subsec:mimo_difficulties}

The results above show that the basic second-order time-domain moment
matching framework extends naturally to MIMO systems once tangential data are
used. However, the stronger SISO results in Sections~\ref{sec:MM_polezero}
and~\ref{sec:firstorder} do not admit a direct MIMO analogue.

First, the derivative-matching result in Theorem~\ref{thm:equivalence}
uses the fact that, in the SISO case, $W'(s_i)$ is a scalar moment at each
interpolation point. In the MIMO case, $W'(s_i)$ is a matrix, and one must
distinguish between the right derivative data $W'(s_i)\ell_i$, the left
derivative data $\rho_i^\ast W'(s_i)$, and the bitangential scalar data
$\rho_i^\ast W'(s_i)\ell_i$. The two-sided projection above guarantees only
the last quantity when the same interpolation point appears in the right and
left data. Matching all one-sided derivative directions, or the full matrix
$W'(s_i)$, requires additional tangential directions and additional
compatibility conditions between the associated Sylvester equations. Moreover,
the cascade realization used in the SISO proof of Theorem~\ref{thm:one2one}
depends on scalar input-output coupling, while in the MIMO case, the dimensions of the
intermediate signals depend on whether one targets right, left, or
bitangential derivative data. A complete parameterization of reduced
second-order models matching $W$ and the desired MIMO derivative data
therefore remains open.

Second, the pole-zero placement result relies on the scalar transfer function
structure. Pole assignment concerns the roots of the reduced quadratic pencil
and can still be formulated for MIMO systems, but the SISO proof uses scalar
moment directions and rank conditions such as $\rank(\Upsilon_{\rm p}\Pi)$ to
obtain linear constraints on the free parameters. In the MIMO case, prescribed poles
must be handled together with possible partial multiplicities and deflating
subspaces of the reduced quadratic pencil, and the resulting constraints are
not a direct copy of the SISO equations.
The obstruction is even more pronounced for zeros. In the SISO case, zeros
can be characterized by a scalar determinant condition. For a MIMO
second-order reduced model, zeros are transmission zeros, i.e., points at
which the Rosenbrock matrix
\[
\mathcal R_{\widehat\Sigma}(z)=
\begin{bmatrix}
z^2F_2+zF_1+F_0 & -G\\
zH_1+H_0 & 0
\end{bmatrix}
\]
drops rank relative to its normal rank. When $p\neq q$, this matrix is not
square, so the determinant condition used in the SISO zero-placement proof is
not available. Even when $p=q$, the zero structure includes input and output
zero directions and multiplicities, and these directions interact with the
moment-matching directions. Consequently, simultaneous MIMO moment matching
and pole-zero placement leads to coupled rank and direction constraints,
rather than the linear scalar constraints obtained in the SISO case.

For these reasons, this section can be viewed as an initial MIMO extension:
the tangential moment characterization, the one-sided and two-sided
second-order reduced families, and the bitangential Hermite property extends
cleanly. A full MIMO theory for first-order derivative matching and for
moment matching with pole-zero placement is an open problem and is left for
future work.

\section{Example} \label{sec:Example}

We verify the effectiveness of the proposed methods on a benchmark second-order system representing a mass-spring-damper chain, shown in Fig.~\ref{fig:msd}. The system consists of $n=200$ masses interconnected by springs and dampers. The coefficient matrices of the second-order system are  given as in \cite{cheng-i-iftime-necoara-CDC2024,simard-moreschini-astolfi-CDC2023}:
\begin{align*}
D & = \left[\begin{matrix}
	c_1 & -c_1 &  & &  \\ 
	-c_1 & c_1 + c_2 & -c_2 & & \\
	& \ddots &  \ddots & \ddots &  & \\
	&   &  \ddots & \ddots & -c_{n-1} \\
	& &  &  -c_{n-1} &  c_{n-1} + c_{n}
\end{matrix}\right],
\\
K & = \left[  \begin{matrix}
	k_1 & -k_1 &  & &  \\ 
	-k_1 & k_1 + k_2 & -k_2 & & \\
	& \ddots &  \ddots & \ddots &  & \\
	&   &  \ddots & \ddots & -k_{n-1} \\
	& &  &  -k_{n-1} &  k_{n-1} + k_{n}
\end{matrix}\right], \\
M &= \text{diag}\big[ m_1, \ \ldots, \ m_{n} \big], 	 B^\top   = C = \begin{bmatrix}
	1 & 0 & \cdots & 0
\end{bmatrix}
\end{align*}
where $m_i$, $k_i$, and $c_i$ are the masses, spring coefficients, and damping coefficients, respectively, for $i = 1, ..., n$. The external input $u$ is the external force acting on the first mass $m_1$, and we measure the displacement of the mass $m_1$ as the output. For simulation, we set $m_i = 1$, $c_i = 2$, and $k_i = 1$ for all $i = 1, \dots, n$.

\begin{figure}[t]
\centering 
\begin{tikzpicture}[scale=0.9]
	\tikzstyle{mass}=[rectangle,  drop shadow,
	minimum width=1cm, 
	minimum height=1cm,
	text centered, 
	fill=black!20]
	\tikzstyle{spring}=[thick,decorate,decoration={zigzag,pre length=0.08cm,post length=0.08cm,segment length=5}]
	\tikzstyle{damper}=[thick,decoration={markings,  
		mark connection node=dmp,
		mark=at position 0.5 with 
		{
			\node (dmp) [thick,inner sep=0pt,transform shape,rotate=-90,minimum width=10pt,minimum height=3pt,draw=none] {};
			\draw [thick] ($(dmp.north east)+(2pt,0)$) -- (dmp.south east) -- (dmp.south west) -- ($(dmp.north west)+(2pt,0)$);
			\draw [thick] ($(dmp.north)+(0,-5pt)$) -- ($(dmp.north)+(0,5pt)$);
		}
	}, decorate]
	\tikzstyle{ground}=[fill,pattern=north east lines,draw=none,minimum width=3mm,minimum height=0.2cm]
	
	\node (M) [mass] {$m_1$};
	\node (M2) [mass] at (2,0) {$m_2$};
	\node (M3) [outer sep=0pt,thick,minimum width=1cm, minimum height=1.6cm] at (4,0) {$...$};
	\node (M4) [mass] at (6,0) {$m_n$};
	\node[ground,minimum width=3mm,minimum height=1.5cm] at (7.7,0) (g1){};
	
	%
	\draw [thick] (M.south west) ++ (0.2cm,-0.125cm);
	\draw [thick] (M2.south west) ++ (0.2cm,-0.125cm);
	\draw [thick] (g1.south west) -- (g1.north west);	
	\draw [thick] (M4.south west) ++ (0.2cm,-0.125cm);
	
	\draw [spring] (M2.210) -- ($(M.north east)!(M.210)!(M.south east)$);
	\draw [spring] (M3.210) -- ($(M2.north east)!(M2.210)!(M2.south east)$);
	\draw [spring] (M4.210) -- ($(M3.north east)!(M3.210)!(M3.south east)$);
	\draw [spring] (g1.240) -- ($(M4.north east)!(M4.210)!(M4.south east)$);
	
	\draw [damper] (M2.160) -- ($(M.north east)!(M2.160)!(M.south east)$);
	\draw [damper] (M3.160) -- ($(M2.north east)!(M3.160)!(M2.south east)$);
	\draw [damper] (M4.160) -- ($(M3.north east)!(M4.160)!(M3.south east)$);
	\draw [damper] (g1.130) -- ($(M4.north east)!(M4.160)!(M4.south east)$);
	
	\node at (1,.8) {$c_1$};
	\node at (3,.8) {$c_2$};
	\node at (5,.8) {$c_{n-1}$};
	\node at (7,.8) {$c_n$};
	
	\node at (1,-.8) {$k_1$};
	\node at (3,-.8) {$k_2$};
	\node at (5,-.8) {$k_{n-1}$};
	\node at (7,-.8) {$k_n$};
	
\end{tikzpicture}
\caption{A mass-spring-damper system with $n$ masses.}
\label{fig:msd}
\end{figure}

We compare the performance of our proposed time-domain second-order moment matching methods for derivative matching (TDMM-FD) and pole placement (TDMM-PP), against two well-established techniques: the second-order balanced truncation (SOBT) \cite{chahlaoui2006balancing2o} and the second-order Arnoldi (SOAR) method \cite{bai2005Arnoldi2o}. For all methods, we construct a reduced-order model of size $r=10$.
The TDMM-based methods are configured to match moments at the same set of $r=10$ interpolation points, $\sigma(S) = \text{diag}(0.01, 0.02, 0.04, ..., 5.12)$. For the TDMM-FD approach, the reduced model $\bm{\widehat \Sigma}_f$ is constructed as in \eqref{eq:RedMats} to match the first-order derivatives at these points. For the TDMM-PP approach, the reduced model $\bm{\widehat \Sigma}_p$ is designed to place a prescribed set of stable poles at $\sigma(\Qp)$, satisfying the conditions of Theorem~\ref{thm:pole}.

A comparison of the relative reduction error in terms of relative $H_\infty$ error and computational time is shown in Table~\ref{tab:mor_comparison_hinf}. TDMM-FD outperforms both SOBT and SOAR in this example, while TDMM-PP has the largest reduction error. This is because of the placement of prescribed poles, which sacrifices certain approximation accuracy. Furthermore, we also observed in the simulations that the accuracy of TDMM-PP is highly sensitive to the choice of prescribed poles. Regarding computational efficiency, SOAR achieves the shortest computation time, and the proposed TDMM methods are significantly faster than SOBT. This is because our approaches require solving a quadratic Sylvester equation of dimension $n$, whereas SOBT requires the solution of two Lyapunov equations for the controllability and observability Gramians of the equivalent first-order system, which has a dimension of $2n$. The frequency responses of the full-order and reduced-order models are shown in Figure~\ref{fig:bode1}, where all methods demonstrate close agreement in the lower-frequency range, while SOAR and SOBT show larger deviations than TDMM-FD in the higher-frequency region.

\begin{table}[t]
\centering
\caption{Performance metrics of model reduction methods.}
\label{tab:mor_comparison_hinf}
\begin{tabular}{l c c c c}
	\toprule
	& \textbf{SOBT} & \textbf{SOAR} & \textbf{TDMM-FD} & \textbf{TDMM-PP} \\
	\midrule
	\textbf{Error ($\times 10^{-3}$)} & 1.809  & 3.790  & 0.905  & 36.639 \\
	\textbf{Time (s)}                          & 0.1467 & 0.0040 & 0.0778 & 0.0981 \\
	\bottomrule
\end{tabular}
\end{table}

\begin{figure}[t]
\centering
\includegraphics[width=0.5\textwidth]{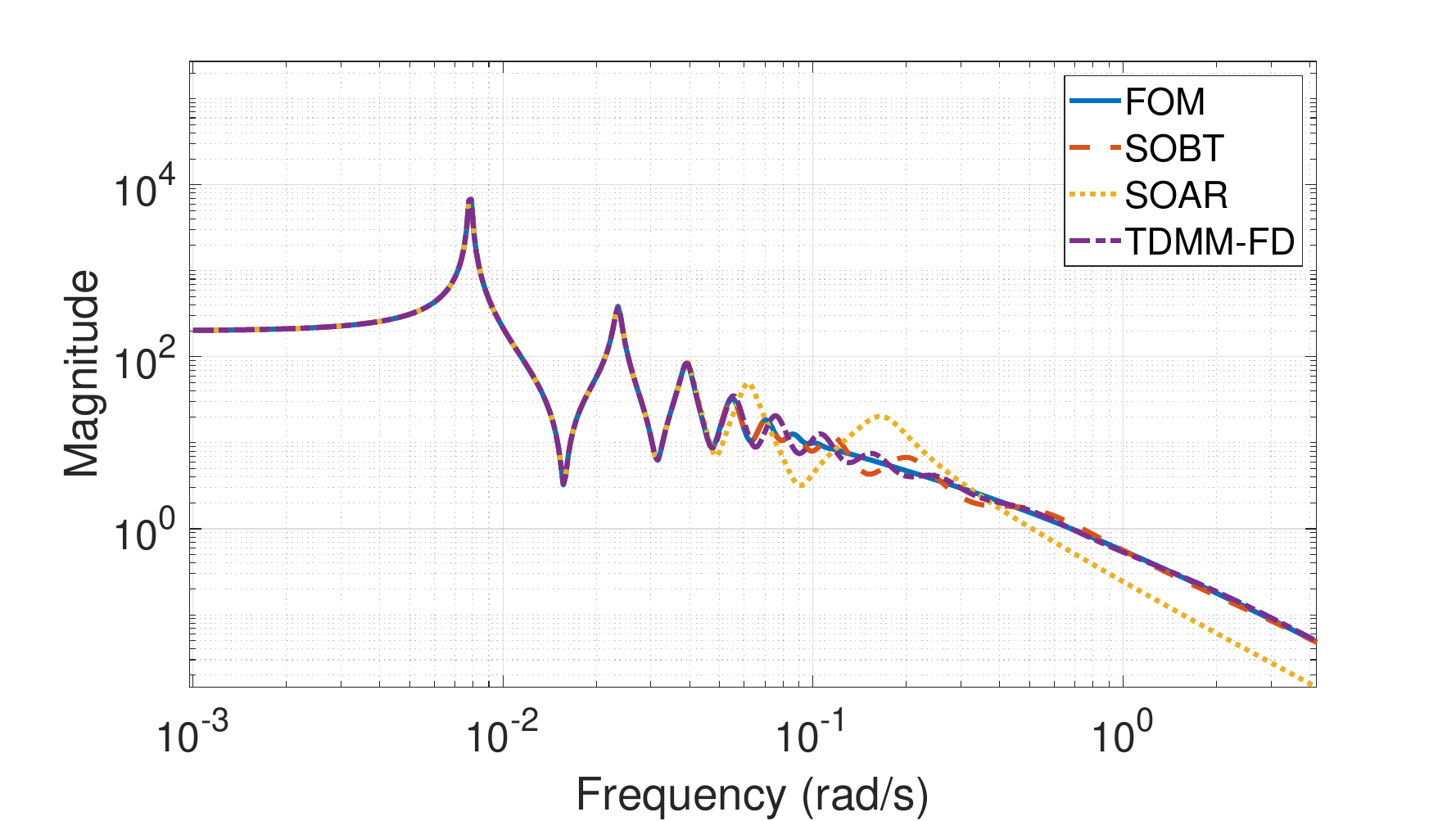} 
\caption{Frequency response comparison of the full-order model (FOM) and the reduced-order models ($r=10$) obtained via SOBT, SOAR, and TDMM-FD.}
\label{fig:bode1}
\end{figure}





\section{Conclusion} 
\label{sec:Conclusion}
A time-domain moment matching framework for second-order dynamical systems has been presented. The moments of a given second-order system are characterized by the unique solution of a second-order Sylvester equation, and families of parameterized reduced second-order models have been provided to match selected moments. Furthermore, we have also determined the free parameters to achieve moment matching at two distinct sets of interpolation points, to achieve moment matching and pole-zero placement, as well as matching the first-order derivative of the transfer function of the original second-order system. However, as noted in \cite{cheng-i-iftime-necoara-CDC2024}, the method has limitations, i.e., it is not adequate when interpolating at pure harmonic frequencies, when the matrices involved may become derogatory.

For MIMO second-order systems, we have shown that the basic Sylvester equation
moment characterization extends naturally to the right and left tangential
moments, leading to one-sided and two-sided tangential reduced models and a
bitangential Hermite interpolation property. A complete MIMO theory for
first-order derivative matching and for moment matching with pole-zero placement remains an open problem and is left for future work.


\section*{Appendices}
\renewcommand{\thesubsection}{\Alph{subsection}}

\subsection{Proof of Lemma \ref{lem:2Omoments1}}
\label{app:Lem:2Omoments1}
Let $\Pi = \left[\Pi_1, \Pi_2, \cdots, \Pi_{\nu} \right]\in \mathbb{R}^{n \times \nu}$ with $\Pi_i \in \mathbb{R}^n$. Then, the matrix equation \eqref{eq:Sylv_nu} is written as
\begin{equation*}
	M \Pi_i s_i^2 + D \Pi_i s_i + K \Pi_i = B l_i,\ \forall i = 1,2,\cdots, \nu.
\end{equation*}
leading to
$ 
\Pi_i = (Ms_i^2 + Ds_i +K)^{-1} B l_i. 
$ 
Thus, $\eta_0(s_i) =  C_0 \Pi_i + C_1 \Pi_i s_i $ for all $i = 1,2,\cdots, \nu$, which gives the result.

Analogously, we denote $\Upsilon^\top: = [\Upsilon_{1}^\top, \Upsilon_{1}^\top, \cdots, \Upsilon_{\nu}^\top]$ with $\Upsilon_{i} \in \mathbb{R}^{1 \times n}$. Then, \eqref{eq:Sylv_nu2} is equivalent to
\begin{equation*}
	\Upsilon_i M s_{\nu+i}^2 + \Upsilon_i D  s_{\nu+i} + \Upsilon_i K  =  r_i^\top C_0 + r_i^\top C_1 s_{\nu+i},
\end{equation*}
for all $i = 1,2, \cdots, \nu$. Thus, we obtain
\begin{equation*}
	\Upsilon_i = (r_i^\top C_0 + r_i^\top C_1 s_{\nu+i}) (Ms_{i+\nu}^2 + Ds_{i+\nu} +K)^{-1},
\end{equation*}
which gives the 0-moments $\eta_0(s_{\nu+1})$, $\cdots$, $\eta_0(s_{2\nu})$.

\subsection{Proof of Lemma \ref{lem:2Omoments2}}
\label{app:Lem:2Omoments2}
For simplicity, let
$\mathscr{F}(s) = (Ms^2 + Ds + K)^{-1}$.
The $k$-th order derivative ($k \geq 2$) of $\mathscr{F}(s)$   is given by
\begin{equation} \label{eq:dkF(s)}
	\begin{split}
		\frac{\opd^k}{\opd s^k} \mathscr{F}(s)  = & -k \mathscr{F}(s)
		\frac{\opd}{\opd s}\mathscr{F}(s)^{-1}
		\frac{\opd^{k-1}}{\opd s^{k-1}}\mathscr{F}(s) \\
		& -\frac{k(k-1)}{2} \mathscr{F}(s) \frac{\opd^2}{\opd s^2}\mathscr{F}(s)^{-1} 
		\frac{\opd^{k-2}}{\opd s^{k-2}}\mathscr{F}(s).
	\end{split}
\end{equation}

We start proving the first statement.
Let $ \Pi = \left[\Pi_0, \Pi_1, \cdots, \Pi_\nu \right]\in \mathbb{R}^{n \times (\nu+1)}$ with $  \Pi_0=(Ms_\star^2 + Ds_\star + K)^{-1}Bl_0$ and 
\begin{equation}\label{eq:barPi_k}
	\Pi_k : =  \frac{1}{k!}\left[ \left.\frac{\opd^k}{\opd s^k}\mathscr{F}(s)\right\rvert_{s = s_\star}\right]  B l_0, \ k=1,2, \cdots, \nu,
\end{equation}
where $l_0 \in \mathbb{R}^p$.
Then, it is not hard to verify from \eqref{eq:dkF(s)} that 
\begin{equation*}
	\begin{split}
		\left(Ms_\star^2  + Ds_\star + K\right){\Pi}_0 =& B l_0, 
		\\
		\left(Ms_\star^2  + Ds_\star + K\right)\Pi_1
		=&-\left(2Ms_\star + D \right) \Pi_0, 		
		\\
		\left(Ms_\star^2  + Ds_\star + K\right) \Pi_2
		=&-\left(2 Ms_\star + D \right) \Pi_1 - M  \Pi_{0}, 
		\\
		\vdots 
		\\
		\left(Ms_\star^2  + Ds_\star + K\right)\Pi_\nu
		=&-\left(2M s_\star +D \right)\Pi_{\nu-1}
		- M \Pi_{\nu-2}.
	\end{split}
\end{equation*}	
The above equations can be rewritten in a compact form: 
\begin{equation} \label{eq:barSylv_k}
	M \Pi \bar{S}^2  + D \Pi \bar{S} + K\Pi = B \bar L,
\end{equation} 
with $\bar L: =\begin{bmatrix}
	l_0 & 0 & ... & 0
\end{bmatrix}$ and 
\begin{equation*}
	\bar S : = \begin{bmatrix}
		s_\star & 1 & 0 & \cdots & 0 \\
		0 & s_\star & 1 & \cdots & 0 \\
		\vdots & \vdots & \ddots & \ddots & \vdots \\
		0 & \cdots & 0 & s_\star & 1 \\
		0 & \cdots & \cdots & 0 & s_\star\\ 
	\end{bmatrix} \in \mathbb{R}^{(\nu+1) \times (\nu+1)}.
\end{equation*}

Next, the moments at $\sigma(S)$ are characterized. The 0-moment is obtained directly as
\begin{equation*}
	\eta_0(s_\star) = ( C_0 + C_1s_\star)\mathscr{F}(s_\star) B = C_0 \Pi_0 + C_1 \Pi_0 s_\star.
\end{equation*}
Furthermore, note that  
\begin{equation*}\label{eq:dkCF(s)}
	\begin{split}
		&\frac{\opd^k}{\opd s^k}\left[(C_0 + C_1s) \mathscr{F}(s)\right]
		\\
		=& 
		kC_1\frac{\opd^{k-1}}{\opd s^{k-1}} \mathscr{F}(s) +  (C_0 + C_1s) \frac{\opd^k}{\opd s^k}\mathscr{F}(s).  
	\end{split}	
\end{equation*}
Thus, by the definition of the $k$-moment in \eqref{eq:k-moment}, we have
\begin{equation*}
	\begin{split}
		\eta_k(s_\star) = & \frac{(-1)^k}{k!} \left.\frac{\opd^k}{\opd s^k}\left[(C_0 + C_1s) \mathscr{F}(s)\right] B \right\rvert_{s = s_\star}\\
		=& (-1)^k \left[ C_0 \Pi_k + C_1(\Pi_{k-1} + \Pi_k s_\star)   \right],
	\end{split}
\end{equation*}
for $k = 1,2,\cdots,\nu$.	
Then, the following relation holds. 
\begin{equation*} \label{eq:one2one}
	\begin{bmatrix}
		\eta_0(s_\star)&
		\eta_1(s_\star)&
		\cdots &
		\eta_\nu(s_\star) 
	\end{bmatrix} 
	= 
	(C_0 \Pi + C_1 \Pi S) \Phi_\nu,
\end{equation*}
with $
\Phi_\nu = \mathrm{diag}(1, -1, 1, \cdots, (-1)^\nu).
$
Therefore, there is a one-to-one relation between the moments $\eta_k(s_\star)$ and the entries of the matrix $C_0 {\Pi} + C_1 {\Pi} S$.

Notice that the pair $(\bar{L}, \bar{S})$ is observable for any $s_\star$. For a given pair $(L, S)$ that is observable, there exists a unique
invertible matrix $T \in \mathbb{R}^{(\nu+1) \times (\nu+1)}$ such that $\bar S = T {S} T^{-1}$ and  $\bar L = T^{-1} {L} $. Substituting $\bar S$ and $\bar L$ into \eqref{eq:barSylv_k} yields the Sylvester equation in \eqref{eq:Sylv_k}. 	

The second statement can be proved following a similar procedure and hence is omitted.

\subsection{Proof of Theorem \ref{thm:Moments}}
\label{app:thm:Moments}

\label{app:thm:Moments}
	It follows from Lemma \ref{lem:2Omoments1} and Lemma \ref{lem:2Omoments2} that the moments of $W(s)$ at $\sigma(S)$ and $\sigma(Q)$ are characterized by
	$C_0 \Pi + C_1 \Pi S$ and $\Upsilon B$, respectively, where $\Pi$ and $\Upsilon$ satisfy the second-order Sylvester equations in  \eqref{eq:Sylv1} and \eqref{eq:Sylv2}, respectively. It remains to show that these second-order Sylvester equations have unique solutions.  
	
	Let
	\[
	P(s):=Ms^2+Ds+K .
	\]
	Since $\sigma(S)\cap\Omega=\emptyset$, the matrix $P(\lambda)$ is nonsingular for every $\lambda\in\sigma(S)$. Take a complex triangular decomposition $S=V_SJ_SV_S^{-1}$, where $J_S$ is upper triangular and has diagonal entries $\lambda_1,\ldots,\lambda_\nu$. Multiplying \eqref{eq:Sylv1} by $V_S$ from the right and setting
	\[
	X:=\Pi V_S,\qquad \bar L:=LV_S,
	\]
	gives
	\begin{equation}\label{eq:quadSylvRightTri}
		MXJ_S^2+DXJ_S+KX=B\bar L.
	\end{equation}
	Write $X=[\chi_1,\ldots,\chi_\nu]$ and $\bar L=[\bar l_1,\ldots,\bar l_\nu]$. Because $J_S$ and $J_S^2$ are upper triangular, the $j$-th column of \eqref{eq:quadSylvRightTri} has the form
	\[
	P(\lambda_j)\chi_j=B\bar l_j-r_j(\chi_1,\ldots,\chi_{j-1}),
	\]
	where $r_j$ depends only on the previously determined columns. Since $P(\lambda_j)$ is nonsingular, $\chi_j$ is uniquely determined. Proceeding recursively for $j=1,\ldots,\nu$ gives a unique $X$, and hence a unique $\Pi=XV_S^{-1}$.
	
	The proof for \eqref{eq:Sylv2} is analogous. Since $\sigma(Q)\cap\Omega=\emptyset$, $P(\mu)$ is nonsingular for every $\mu\in\sigma(Q)$. Let $Q=V_QJ_QV_Q^{-1}$, where $J_Q$ is upper triangular with diagonal entries $\mu_1,\ldots,\mu_\nu$. Multiplying \eqref{eq:Sylv2} by $V_Q^{-1}$ from the left and setting
	\[
	Y:=V_Q^{-1}\Upsilon,\qquad \bar R:=V_Q^{-1}R,
	\]
	gives
	\begin{equation}\label{eq:quadSylvLeftTri}
		J_Q^2YM+J_QYD+YK=\bar R C_0+J_Q\bar R C_1.
	\end{equation}
	Write the rows of $Y$ as $\eta_1,\ldots,\eta_\nu$. Since $J_Q$ and $J_Q^2$ are upper triangular, the $i$th row of \eqref{eq:quadSylvLeftTri} has the form
	\[
	\eta_iP(\mu_i)=\bar r_i(C_0+\mu_iC_1)-s_i(\eta_{i+1},\ldots,\eta_\nu),
	\]
	where $s_i$ depends only on rows with larger indices. As $P(\mu_i)$ is nonsingular, $\eta_i$ is uniquely determined. Solving backward for $i=\nu,\ldots,1$ gives a unique $Y$, and hence a unique $\Upsilon=V_QY$. Therefore, $\Pi$ and $\Upsilon$ are uniquely determined, and the theorem follows.
	
	\subsection{Proof of Proposition \ref{pro:interconnection}}
	\label{app:pro:interconnection}
		From \eqref{geneQR}, we have
		\begin{align} \label{eq:dot_d}
			\dot{d} = \dot{\varpi} + (Q\Upsilon M + \Upsilon D - RC_1) \dot{x} + \Upsilon M \ddot{x},
		\end{align}
		where $\dot{\varpi} = Q\varpi + R(C_1 \dot{x} + C_0 x)$. Moreover, on the manifold $\mathcal{M}$, it holds that $\dot{x} = \Pi \dot{\omega} = \Pi S \omega$, and $\ddot{x} = \Pi S \dot{\omega} = \Pi S^2 \omega$, which are substituted into \eqref{eq:dot_d} and lead to 
		\begin{align} \label{eq:dot_d2}
			\dot{d}  = Q\varpi + \left[RC_0 \Pi + (Q\Upsilon M + \Upsilon D ) \Pi S  + \Upsilon M \Pi S^2\right] \omega.
		\end{align} 
		Observe that \eqref{eq:Sylv1} and \eqref{eq:Sylv2} imply that
		\begin{align*} 
			&RC_0 \Pi +  \Upsilon D \Pi S + \Upsilon M \Pi S^2
			\\ = & Q^2 \Upsilon M \Pi +  Q  \Upsilon D \Pi + \Upsilon B L - Q R_1 \Pi 
			\\= &Q(Q\Upsilon M + \Upsilon D - RC_1) \Pi + \Upsilon B L.
		\end{align*} 
		Consequently, \eqref{eq:dot_d2} is further written as
		\begin{align*}  
			\dot{d} & = Q \left[\varpi + (Q\Upsilon M + \Upsilon D - RC_1) x + \Upsilon M \dot{x}\right] + \Upsilon BL \omega\\
			& = Q d +  \Upsilon BL \omega,
		\end{align*} 
		which completes the proof.
		
		\subsection{Proof of Theorem \ref{thm:twoside}}
		\label{app:thm:twoside}
			We start with the proof for $\bm{\widehat{\Sigma}}_G$. With the parameters in \eqref{eq:redMats}, we obtain
			$F_0 = \Pi^\dagger K$, and $H_0 = C_0 \Pi$.	
			Consider a system $\bm{\widehat{\Sigma}}_G$ in form of \eqref{sys:2orG}, which connects the signal generator 
			\begin{align*} 
				&\dot{w} = Q w + R \widehat\psi, w(0) = 0, \nonumber \\ 
				&\zeta = w + (Q P F_2 + P F_1 - RH_1) \xi + P F_2 \dot{\xi}.
			\end{align*}
			as a downstream system with $\widehat \psi = \eta$. Then, the system $\bm{\widehat{\Sigma}}_G$ matches the moments $\Upsilon B$, with $\Upsilon$ the unique solution of \eqref{eq:Sylv2}, at $\sigma(Q)$ if and only if 
			\begin{align} \label{eq:dot_zeta1}
				\dot{\zeta} & = Q \zeta + \Upsilon B u \nonumber\\
				& = Qw + (Q^2 P F_2 + Q P F_1 - Q RH_1) \xi + Q P F_2 \dot{\xi} + \Upsilon B u.
			\end{align}
			We refer to \cite{Astolfi2010MMCDC,Ionescu2016TwoSided} for similar reasoning in the case of first-order systems. Note that
			\begin{align}\label{eq:dot_zeta2}
				\dot{\zeta} = & \dot{w} + (Q P F_2 + P F_1 - RH_1) \dot{\xi} + P F_2 \ddot{\xi} 
				\nonumber\\
				= 	& Q w + R \left[H_1 \dot{\xi} + (C_0 \Pi + C_1 \Pi S  - H_1S) \xi\right] \nonumber \\
				&-RH_1 \dot{\xi} + Q P F_2 \dot{\xi} + P F_1 \dot{\xi} + P F_2 \ddot{\xi} 
				\nonumber \\
				= & Q w + R(C_0 \Pi + C_1 \Pi S  - H_1S)  \xi \nonumber\\
				& + Q P F_2 \dot{\xi} + PGu - P (GL - F_2 S^2 - F_1 S) \xi. 
			\end{align}
			Therefore, from \eqref{eq:dot_zeta1} and \eqref{eq:dot_zeta2}, the system $\bm{\widehat{\Sigma}}_G$ matches the moments $\Upsilon B$, if and only if the parameters $F_1$, $F_2$, $G$ and $H_1$ in $\bm{\widehat{\Sigma}}_G$ satisfy
			\begin{equation*} \label{eq:PGYB}
				P G = \Upsilon B, 
			\end{equation*}
			and
			\begin{align} \label{eq:sylvthm}
				Q^2 P F_2 + Q P F_1 + P (GL - F_2 S^2 - F_1 S) 
				\nonumber\\ = R(C_0 \Pi + C_1 \Pi S  - H_1S) + Q RH_1.
			\end{align}
			It is verified that $P = \Upsilon \Pi$ is the unique solution of \eqref{eq:sylvthm} due to $\widehat{\Omega} \cap \sigma(Q) = \emptyset$. Moreover, since $\Upsilon$ and $\Pi$ are unique, the parameter matrices of $\bm{\widehat{\Sigma}}_G$ in \eqref{eq:redMats}   is unique.
			
			
			The proof for $\bm{\widehat{\Sigma}}_H$ with parameters in \eqref{eq:redMats2}  follows similar arguments. Besides, the equivalence of  $\bm{\widehat{\Sigma}}_G$ and $\bm{\widehat{\Sigma}}_H$ follows from the nonsingularity of $\Upsilon \Pi$, with which there exists a coordinate transformation between the two systems.
		
		\subsection{Proof of Theorem \ref{thm:pole}}
		\label{app:thm:pole}
			Observe that $\widehat{\Omega}$ of the reduced model $\bm{\widehat{\Sigma}}_G$ is characterized by the solution of the following determinant equation
			\begin{align*}
				|\rho(\lambda)| =  |\lambda^2 F_2 + \lambda F_1 + (GL - F_2 S^2 - F_1 S)| 
				= 0.
			\end{align*}
			With the equations in \eqref{eq:constraints}, the matrix polynomial in the above determinant can be rewritten as
			\begin{align} \label{eq:poleprove1}
				\Upsilon_\mathrm{p} \Pi \rho(\lambda) =&\Upsilon_\mathrm{p} \Pi  \left[ \lambda^2  F_2 + \lambda  F_1 +  (GL -  F_2 S^2 -  F_1 S)\right]
				\nonumber \\
				= & 
				\lambda^2  \Upsilon_\mathrm{p} M \Pi + \lambda  \Upsilon_\mathrm{p} D \Pi 
				\nonumber\\& +  \Upsilon_\mathrm{p} BL - \Upsilon_\mathrm{p} M \Pi   S^2 -  \Upsilon_\mathrm{p} D \Pi  S.
			\end{align}
			Moreover, it follows from the second-order Sylvester equation \eqref{eq:Sylv1} that
			\begin{align}\label{eq:poleprove2}
				& \Upsilon_\mathrm{p} BL - \Upsilon_\mathrm{p} M \Pi   S^2 -  \Upsilon_\mathrm{p} D \Pi  S 
				\nonumber\\
				= 	&
				\Upsilon_\mathrm{p} (M \Pi S^2 + D \Pi S + K \Pi) - \Upsilon_\mathrm{p} M \Pi   S^2 -  \Upsilon_\mathrm{p} D \Pi  S 
				\nonumber\\ = & \Upsilon_\mathrm{p} K \Pi.
			\end{align}
			Let \eqref{eq:Sylv_Qp} be post-multiplied by $\Pi$, which yields
			\begin{equation}\label{eq:poleprove3}
				\Upsilon_\mathrm{p} K \Pi =  - Q_\mathrm{p}^2 \Upsilon_\mathrm{p} M \Pi  -   Q_\mathrm{p} \Upsilon_\mathrm{p} D \Pi,
			\end{equation}
			as $C_{\mathrm{p}0}$ and $C_{\mathrm{p}1}$ are chosen such that $C_{\mathrm{p}0} \Pi = C_{\mathrm{p}1} \Pi = 0$.
			
			Combining \eqref{eq:poleprove1}, \eqref{eq:poleprove2}, and \eqref{eq:poleprove3}, we obtain
			\begin{align} \label{eq:poleprove_rho}
				\Upsilon_\mathrm{p} \Pi \rho(\lambda)  = (\lambda I - Q_\mathrm{p})  
				\left[(\lambda I + Q_\mathrm{p}) \Upsilon_\mathrm{p} M \Pi +  \Upsilon_\mathrm{p} D \Pi\right].
			\end{align}
			
			Notice that $\lambda \in \sigma(Q_\mathrm{p})$ if and only if there exists a left eigenvector $v$ such that $v^\top (\lambda I - Q_\mathrm{p}) = 0$.
			Then, we obtain from \eqref{eq:poleprove_rho} that 
			\begin{align*}
				v^\top (\lambda I - Q_\mathrm{p})  
				\left[(\lambda I + Q_\mathrm{p}) \Upsilon_\mathrm{p} M \Pi +  \Upsilon_\mathrm{p} D \Pi\right]  = 0
			\end{align*}
			i.e. $v^\top \Upsilon_\mathrm{p} \Pi \rho(\lambda)  = \bar{v}^\top  \rho(\lambda) = 0 $ with $\bar{v} = (\Upsilon_\mathrm{p} \Pi)^\top v \in \mathbb{C}^{\nu}$. It means that there is a vector $\bar{v}_r \in \mathbb{C}^{\nu}$ such that $\rho(\lambda) \bar{v}_r = 0$, i.e. 
			$F_2^{-1} ( \lambda^2 F_2 + \lambda F_1 + (GL - F_2 S^2 - F_1 S)) =0 $,
			which 
			is equivalent to  
			\begin{equation*}
				\left(	\lambda I
				-
				\begin{bmatrix}
					0 &  I \\
					- F_2^{-1}(GL - F_2 S^2 - F_1 S) & - F_2^{-1} F_1
				\end{bmatrix}\right)
				\begin{bmatrix}
					\bar{v}_r 
					\\
					\lambda  \bar{v}_r
				\end{bmatrix}
				= 0,
			\end{equation*}
			Therefore, for any $\lambda \in \sigma(Q_\mathrm{p})$, we have $|\rho(\lambda)| = 0$, i.e. $\lambda \in \widehat{\Omega}$.
			
			%

			\subsection{Proof of Theorem \ref{thm:zero}}
			\label{app:thm:zero}
			%
			Let $z\in\sigma(\Qz)$. Then, there exists $w\in\mathbb C^{m_\z}, w\ne 0$ such that $w^\top(zI-\Qz)=0$. Note that, if \eqref{eq_zeros_assign_H1} holds, then \eqref{eq_assign_zeros} also holds if, for instance,
			{\small\begin{align*}
					\begin{bmatrix}
						w^\top\Uz\Pi & w^\top\Rz
					\end{bmatrix} &\begin{bmatrix}
						F_2z_i^2+F_1z_i+(GL-F_2S^2-F_1S) &  -G \\ H_1z_i+(C_0\Pi+C_1\Pi S-H_1S) & 0
					\end{bmatrix} \\ &=0.
			\end{align*}}
			Equivalently, 
			\begin{align*}
				w^\top\Uz\Pi(z^2F_2 & +zF_1+GL-F2S^2-F_1S) \\ 
				&+w^\top\Rz(C_1\Pi z+C_0\Pi)=0.
			\end{align*}
			Hence,
			\begin{align*}
				&z^2w^\top\Uz\Pi F_2-zw^\top(\Uz\Pi F_1+RC_1\Pi) \\
				&+W^\top(\Uz\Pi GL-\Uz\Pi F_2S^2-\Uz\Pi F_1S+RC_0\Pi)=0.
			\end{align*}
			Since $zw^\top=w^\top\Qz$ and performing some calculations, one can equivalently write
			\begin{align}\label{eq_Q}
				&w^\top\Qz^2\Uz\Pi F_2+w^\top\Qz\Uz\Pi F_1\nonumber \\
				&+w^\top(\Qz RC_1\Pi+RC_0\Pi) \\
				&+w^\top(\Uz\Pi GL-\Uz\Pi FS^2-\Uz\Pi F_1S)=0. \nonumber
			\end{align}
			Multiplying \eqref{eq:Sylv1} with $\Uz$ to the left and employing $\Uz B L=0$ yields that
			\begin{align*}
				\Uz K\Pi=-(\Uz M\Pi S^2+\Uz D\Pi S).
			\end{align*}
			Moreover, multiplying \eqref{eq:Sylv_Uz} with $\Pi$ to the right and substituting $\Uz K\Pi$ yields that
			\begin{align}\label{eq_RCPi}
				QRC_1\Pi+RC_0\Pi= Q^2\Uz M\Pi&+Q\Uz D\Pi \\ &-(\Uz M\Pi S^2+\Uz D\Pi S).\nonumber
			\end{align}
			
			Substituting \eqref{eq_RCPi} into \eqref{eq_Q} yields
			
			
			\begin{align*}
				&w^\top\Qz^2\Uz\Pi F_2+w^\top\Qz\Uz\Pi F_1 \\
				&+w^\top(Q^2\Uz M\Pi+\Qz\Uz D\Pi-\Uz M\Pi S^2-\Uz D\Pi S) \\
				&+w^\top(\Uz\Pi GL-\Uz\Pi F_2S^2-\Uz\Pi F_1 S)=0.
			\end{align*}
			
			Taking the terms $w^\top\Qz^2, w^\top\Qz$ and $w^\top$ as common factors, one further writes
			
			\begin{align*}
				w^\top\Qz^2(\Uz\Pi F_2&+\Uz M\Pi)+w^\top\Qz(\Uz\Pi F_1+\Uz D\Pi)  \\
				&+w^\top\left(\Uz\Pi GL-\Uz M\Pi S^2-\Uz D\Pi S\right. \\
				&\left.-\Uz\Pi F_2S^2-\Uz\Pi F_1 S\right)=0,
			\end{align*}
			which holds if relations \eqref{eq_zeros_assign_F2}-\eqref{eq_zeros_assign_G} are satisfied.
			
			\subsection{Proof of Theorem \ref{thm:one2one}}
			\label{app:thm:one2one}
				The scheme of proving this result follows similar arguments as in \cite{Astolfi2010MMCDC,Ionescu2016TwoSided}, but the details require nontrivial modifications due to the second-order structure of the system. Observe that
				\begin{align*}
					\dot{d} & = \dot{\varpi} + (S\Upsilon M + \Upsilon D) \dot{z} + \Upsilon M \ddot{z} 
					\nonumber \\
					& = S \varpi - L^\top C z + S \Upsilon M \dot{z} + \Upsilon(\widetilde{u} - Kz),
				\end{align*}
				which, using $\widetilde{u} = \widetilde{y} = 2 M \dot{x} + D x$ and  \eqref{eq:Sylv12}, leads to
				\begin{align*}
					\dot{d} 
					= & S \varpi - L^\top C z + S \Upsilon M \dot{z} + 2 \Upsilon M \dot{x} + 2 \Upsilon D x \nonumber\\ 
					&  - (-L^\top C - S^2 \Upsilon M - S \Upsilon D)z, 
					\nonumber\\
					= & S \left[ \varpi + (S\Upsilon M + \Upsilon D) {z} + \Upsilon M \dot{z} \right] + \Upsilon (2 M \dot{x} + D x)
					\nonumber \\
					= & S d + \Upsilon \begin{bmatrix}
						D & 2M
					\end{bmatrix} \begin{bmatrix}
						x(t) \\ \dot{x}(t)
					\end{bmatrix},
				\end{align*}	
				Since 
				$M (\ddot{x} - \Pi \ddot{\omega}) + D (\dot{x} - \Pi \dot{\omega}) + K ({x} - \Pi {\omega})
				=  M \ddot{x} + D \dot{x} + K {x} - (M\Pi S^2  + D \Pi S  + K \Pi {\omega}) \omega
				=  B u - B L \omega = 0,
				$
				we obtain
				\begin{equation*}
					\begin{bmatrix}
						x(t) \\ \dot{x}(t)
					\end{bmatrix} = 
					\begin{bmatrix}
						\Pi \\ \Pi S
					\end{bmatrix} \omega +
					\e^{\mathcal{A}t}
					\begin{bmatrix}
						x(0) - \Pi \omega(0) \\
						\dot{x}(0) - \Pi S \omega(0) \\
					\end{bmatrix},
				\end{equation*}
				which yields 
				\begin{align*}
					\dot{d} 
					= & S d + (2\Upsilon M \Pi S + \Upsilon D \Pi) \omega + \Delta_d,
				\end{align*}	
				with $\Delta_d: = \Upsilon \begin{bmatrix}
					D & 2M
				\end{bmatrix}\e^{\mathcal{A}t}
				\begin{bmatrix}
					x(0) - \Pi \omega(0) \\
					\dot{x}(0) - \Pi S \omega(0) \\
				\end{bmatrix}$.
				Denote $\mathscr{D}(s)$ as the Laplace transform of $d(t)$. Note that the term $\Delta_d$ vanishes in the steady-state response, and thus
				\begin{equation*}
					\mathscr{D}(s) = (sI - S)^{-1} \Upsilon (2 M \Pi S + D \Pi) \mathscr{W}(s),
				\end{equation*}
				where $\mathscr{W}(s)$ denotes the Laplace transform of $\omega(t)$. Moreover,  we obtain from  \eqref{geneSL} that $\mathscr{W}(s) : = (sI - S)^{-1} \omega(0)$, which leads to
				\begin{align*}
					\mathscr{D}(s) = (sI - S)^{-1} \Upsilon (2 M s + D )\Pi (sI - S)^{-1} \omega(0) \nonumber\\
					- 2(sI - S)^{-1} \Upsilon M \Pi \omega(0).
				\end{align*}
				Denote
				$\mathscr{F}(s) = (Ms^2 + Ds + K)^{-1}$. Then, we have
				\begin{align*}
					&(sI - S)^{-1} \Upsilon \mathscr{F}^{-1} 
					\\
					= & (sI - S)^{-1} (\Upsilon M s^2 + \Upsilon D s + \Upsilon K) 
					\\
					= & (sI - S)^{-1} (\Upsilon M s^2 + \Upsilon D s - L^\top C - S^2 \Upsilon M - S\Upsilon D)
					\\
					= & (sI + S) \Upsilon M + \Upsilon D - L^\top C,
				\end{align*}
				and similarly,
				\begin{align*}
					& \mathscr{F}^{-1} \Pi (sI - S)^{-1} 
					=   M \Pi (sI + S) + D \Pi + BL.
				\end{align*}
				Therefore, we can rewrite the first term of $\mathscr{D}(s)$ as 	
				\begin{align*}
					\mathscr{D}_1(s) &= (sI - S)^{-1} \Upsilon \mathscr{F}^{-1}(s) \mathscr{F}(s) (2 M s + D ) 
					\nonumber\\
					& \quad \quad \quad \quad \cdot \mathscr{F}(s) \mathscr{F}^{-1}(s) \Pi (sI - S)^{-1} \omega(0),
					\nonumber\\
					& = \left[(sI + S) \Upsilon M + \Upsilon D - L^\top C\right] \mathscr{F}(s) (2 M s + D ) 
					\nonumber\\
					& \quad \quad \quad \quad \cdot \mathscr{F}(s) \left[M \Pi (sI + S) + D \Pi + BL \right]\omega(0) 
					\nonumber\\
					&  = -L^\top C \mathscr{F}(s) (2 M s + D) \mathscr{F}(s) B L + \cdots
				\end{align*}
				Hence, the steady-state response of $d(t)$ contains terms of the form 
				$
				L^\top \frac{W(s_i)}{(s-s_i)^2} L, 
				$
				with $s_i \in \sigma(S)$, proving the claim. 
			
			\subsection{Proof of Theorem \ref{thm:equivalence}}	
			\label{app:thm:equivalence}
				First, with the reduced matrices in \eqref{eq:RedMats}, we obtain $F_0 = \Pi^\dagger K$ and $H = C \Pi$. It is not hard to verify, according to Proposition~\ref{pro:family}, that $\bm{\widehat \Sigma}_G$ matches the moments of $W(s)$ at $\sigma(S)$. Then, we prove that $\bm{\widehat \Sigma}_G$ also matches the moments of $W^\prime(s)$, which means that $W_L(s_i) W_R(s_i) = - H (F_2 s_i^2 + F_1 s_i + F_0)^{-1} (2 F_2s_i + F_1) (F_2 s_i^2 + F_1 s_i + F_0)^{-1} B$, for all $s_i \in \sigma(S)$,
				with the transfer function $W_L(s)$, $W_R(s)$ defined in \eqref{eq:WLWR}. 
				
				Observe that $2 F_2s_i + F_1 = P \Upsilon (2 M s_i + D) \Pi$,
				where $P: = (\Upsilon \Pi)^{-1}$. Therefore, the moment matching is achieved if 
				\begin{equation} \label{eq:match1}
					C (M s_i^2 + D s_i + K)^{-1} = C \Pi (F_2 s_i^2 + F_1 s_i + F_0)^{-1} P \Upsilon,
				\end{equation}
				and 
				\begin{equation}\label{eq:match2}
					(M s_i^2 + D s_i + K)^{-1} B = \Pi (F_2 s_i^2 + F_1 s_i + F_0)^{-1} P \Upsilon B.
				\end{equation}
				It follows from the second-order Sylvester equations \eqref{eq:Sylv1} and \eqref{eq:Sylv2} that
				\begin{align*}
					S^2 P^{-1} F_2 + S P^{-1} F_1 + P^{-1} F_0 = L^\top C \Pi, 
					\\F_2 S^2 + F_1 S + F_0 = P \Upsilon B L,
				\end{align*}
				with $F_2$, $F_1$, and $F_0$ in \eqref{eq:RedMats}. Thus, \eqref{eq:match1} and \eqref{eq:match2} are satisfied.
				
				Besides, we note that the systems  $\bm{\widehat{\Sigma}}_G$ and $\bm{\widehat{\Sigma}}_H$ are equivalent, as there exists a coordinate transformation between the two systems due to the non-singularity of $\Upsilon \Pi$.

			\bibliographystyle{IEEEtran}
			\bibliography{MomentMatching,TCI_articles,TCI_books,TCI_phdthesis}

\begin{thebibliography}{10}
\providecommand{\url}[1]{#1}
\csname url@samestyle\endcsname
\providecommand{\newblock}{\relax}
\providecommand{\bibinfo}[2]{#2}
\providecommand{\BIBentrySTDinterwordspacing}{\spaceskip=0pt\relax}
\providecommand{\BIBentryALTinterwordstretchfactor}{4}
\providecommand{\BIBentryALTinterwordspacing}{\spaceskip=\fontdimen2\font plus
\BIBentryALTinterwordstretchfactor\fontdimen3\font minus
  \fontdimen4\font\relax}
\providecommand{\BIBforeignlanguage}[2]{{%
\expandafter\ifx\csname l@#1\endcsname\relax
\typeout{** WARNING: IEEEtran.bst: No hyphenation pattern has been}%
\typeout{** loaded for the language `#1'. Using the pattern for}%
\typeout{** the default language instead.}%
\else
\language=\csname l@#1\endcsname
\fi
#2}}
\providecommand{\BIBdecl}{\relax}
\BIBdecl

\bibitem{yan2008second}
B.~Yan, S.~X.-D. Tan, and B.~McGaughy, ``Second-order balanced truncation for
  passive-order reduction of {RLCK} circuits,'' \emph{IEEE Tran. on Circ. and
  Sys. II: Express Briefs}, vol.~55, no.~9, pp. 942--946, 2008.

\bibitem{Dorier2014PowerNetworks}
F.~D{\"o}rfler, M.~R. Jovanovic, M.~Chertkov, and F.~Bullo,
  ``Sparsity-promoting optimal wide-area control of power networks,''
  \emph{IEEE Transactions on Power Systems}, vol.~29, no.~5, pp. 2281--2291,
  2014.

\bibitem{Safaee2021SecondOrder}
B.~Safaee and S.~Gugercin, ``Structure-preserving model reduction of parametric
  power networks,'' in \emph{2021 American Control Conference (ACC)}.\hskip 1em
  plus 0.5em minus 0.4em\relax IEEE, 2021, pp. 1824--1829.

\bibitem{Morzfeld2010Vibration}
F.~Ma, M.~Morzfeld, and A.~Imam, ``The decoupling of damped linear systems in
  free or forced vibration,'' \emph{Journal of Sound and vibration}, vol. 329,
  no.~15, pp. 3182--3202, 2010.

\bibitem{Koutsovasilis2008comparison}
P.~Koutsovasilis and M.~Beitelschmidt, ``Comparison of model reduction
  techniques for large mechanical systems,'' \emph{Multibody System Dynamics},
  vol.~20, no.~2, pp. 111--128, 2008.

\bibitem{cheng2023optimal}
X.~Cheng, ``An optimal projection framework for structure-preserving model
  reduction of linear systems,'' \emph{arXiv preprint arXiv:2302.08627}, 2023.

\bibitem{meyer1996balancing2o}
D.~G. Meyer and S.~Srinivasan, ``Balancing and model reduction for second-order
  form linear systems,'' \emph{IEEE Transactions on Automatic Control},
  vol.~41, no.~11, pp. 1632--1644, 1996.

\bibitem{chahlaoui2006balancing2o}
Y.~Chahlaoui, D.~Lemonnier, A.~Vandendorpe, and P.~Van~Dooren, ``Second-order
  balanced truncation,'' \emph{Linear Algebra and its Applications}, vol. 415,
  no.~2, pp. 373--384, 2006.

\bibitem{reis2008BT2O}
T.~Reis and T.~Stykel, ``Balanced truncation model reduction of second-order
  systems,'' \emph{Mathematical and Computer Modelling of Dynamical Systems},
  vol.~14, no.~5, pp. 391--406, 2008.

\bibitem{benner2011efficient2O}
P.~Benner and J.~Saak, ``Efficient balancing-based {MOR} for large-scale
  second-order systems,'' \emph{Math. and Computer Model. of Dyn. Sys.},
  vol.~17, no.~2, pp. 123--143, 2011.

\bibitem{benner2013improved2O}
P.~Benner, P.~K{\"u}rschner, and J.~Saak, ``An improved numerical method for
  balanced truncation for symmetric second-order systems,'' \emph{Math. and
  Computer Model. of Dyn. Sys.}, vol.~19, no.~6, pp. 593--615, 2013.

\bibitem{hartmann2010BT2O}
C.~Hartmann, V.-M. Vulcanov, and C.~Sch{\"u}tte, ``Balanced truncation of
  linear second-order systems: a hamiltonian approach,'' \emph{Multiscale
  Modeling \& Simulation}, vol.~8, no.~4, pp. 1348--1367, 2010.

\bibitem{Sato2017Riemannian}
K.~Sato, ``Riemannian optimal model reduction of linear second-order systems,''
  \emph{IEEE Control Systems Letters}, vol.~1, no.~1, pp. 2--7, 2017.

\bibitem{Yu2021h_2}
L.~Yu, X.~Cheng, J.~M.~A. Scherpen, and J.~Xiong, ``{$H_2$} model reduction for
  diffusively coupled second-order networks by convex-optimization,''
  \emph{Automatica}, vol. 137, p. 110118, 2022.

\bibitem{XiaodongTAC20172OROM}
X.~Cheng, Y.~Kawano, and J.~M.~A. Scherpen, ``Reduction of second-order network
  systems with structure preservation,'' \emph{IEEE Transactions on Automatic
  Control}, vol.~62, pp. 5026 -- 5038, 2017.

\bibitem{cheng2016secondclustering}
X.~Cheng, J.~M.~A. Scherpen, and Y.~Kawano, ``Model reduction of second-order
  network systems using graph clustering,'' in \emph{2016 IEEE 55th Conference
  on Decision and Control (CDC)}, 2016, pp. 7471--7476.

\bibitem{XiaodongACOM2018Power}
X.~Cheng and J.~M.~A. Scherpen, ``Clustering approach to model order reduction
  of power networks with distributed controllers,'' \emph{Advances in
  Computational Mathematics}, vol.~44, no.~6, pp. 1917--1939, Dec 2018.

\bibitem{Ishizaki2015clustered}
T.~Ishizaki and J.-i. Imura, ``Clustered model reduction of interconnected
  second-order systems,'' \emph{Nonlinear Theory and Its Applications, IEICE},
  vol.~6, no.~1, pp. 26--37, 2015.

\bibitem{antoulas2005approximation}
A.~C. Antoulas, \emph{Approximation of Large-Scale Dynamical Systems}.\hskip
  1em plus 0.5em minus 0.4em\relax Philadelphia, USA: SIAM, 2005.

\bibitem{gallivan2004sylvester}
K.~Gallivan, A.~Vandendorpe, and P.~Van~Dooren, ``Sylvester equations and
  projection-based model reduction,'' \emph{Journal of Computational and
  Applied Mathematics}, vol. 162, no.~1, pp. 213--229, 2004.

\bibitem{antoulas2010interpolatory}
A.~C. Antoulas, C.~A. Beattie, and S.~Gugercin, ``Interpolatory model reduction
  of large-scale dynamical systems,'' in \emph{Efficient modeling and control
  of large-scale systems}.\hskip 1em plus 0.5em minus 0.4em\relax Springer,
  2010, pp. 3--58.

\bibitem{Astolfi2010MM}
A.~Astolfi, ``Model reduction by moment matching for linear and nonlinear
  systems,'' \emph{IEEE Tran. on Aut. Contr.}, vol.~55, no.~10, pp. 2321--2336,
  2010.

\bibitem{astolfi2020model}
A.~Astolfi, G.~Scarciotti, J.~Simard, N.~Faedo, and J.~V. Ringwood, ``Model
  reduction by moment matching: Beyond linearity a review of the last 10
  years,'' in \emph{2020 59th IEEE CDC}.\hskip 1em plus 0.5em minus 0.4em\relax
  IEEE, 2020, pp. 1--16.

\bibitem{bai2005Arnoldi2o}
Z.~Bai and Y.~Su, ``Dimension reduction of large-scale second-order dynamical
  systems via a second-order {Arnoldi} method,'' \emph{SIAM Journal on
  Scientific Computing}, vol.~26, no.~5, pp. 1692--1709, 2005.

\bibitem{salimbahrami2order}
B.~Salimbahrami and B.~Lohmann, ``Order reduction of large scale second-order
  systems using {Krylov} subspace methods,'' \emph{Linear Algebra and its
  Applications}, vol. 415, no.~2, pp. 385--405, 2006.

\bibitem{beattie2005Krylov2o}
C.~A. Beattie and S.~Gugercin, ``Krylov-based model reduction of second-order
  systems with proportional damping,'' in \emph{Proceedings of the 44th IEEE
  CDC-ECC 2005}.\hskip 1em plus 0.5em minus 0.4em\relax IEEE, 2005, pp.
  2278--2283.

\bibitem{qiu2018interpolatory2O}
Z.-Y. Qiu, Y.-L. Jiang, and J.-W. Yuan, ``Interpolatory model order reduction
  method for second order systems,'' \emph{Asian Journal of Control}, vol.~20,
  no.~1, pp. 312--322, 2018.

\bibitem{Vakilzadeh2018krylov}
M.~Vakilzadeh, M.~Eghtesad, R.~Vatankhah, and M.~Mahmoodi, ``A krylov subspace
  method based on multi-moment matching for model order reduction of
  large-scale second order bilinear systems,'' \emph{Applied Mathematical
  Modelling}, vol.~60, pp. 739--757, 2018.

\bibitem{beeumen-nimmen-lombaert-meerbergen-IJNME2012}
R.~V. Beeumen, K.~V. Nimmen, G.~Lombaert, and K.~Meerbergen, ``Model reduction
  for dynamical systems with quadratic output,'' \emph{Int. J. for Numerical
  Methods in Eng.}, vol.~91, pp. 229--248, 2012.

\bibitem{Astolfi2010MMCDC}
A.~Astolfi, ``Model reduction by moment matching, steady-state response and
  projections,'' in \emph{Proceedings of 49th IEEE Conference on Decision and
  Control (CDC)}.\hskip 1em plus 0.5em minus 0.4em\relax IEEE, 2010, pp.
  5344--5349.

\bibitem{Ionescu2013AUT}
T.~C. Ionescu and A.~Astolfi, ``Families of moment matching based, structure
  preserving approximations for linear port-{Hamiltonian} systems,''
  \emph{Automatica}, vol.~49, no.~8, pp. 2424--2434, 2013.

\bibitem{Ionescu2016TwoSided}
T.~C. Ionescu, ``Two-sided time-domain moment matching for linear systems.''
  \emph{IEEE Tran. Aut. Contr.}, vol.~61, no.~9, pp. 2632--2637, 2016.

\bibitem{Ionescu2014SCL}
T.~C. Ionescu, A.~Astolfi, and P.~Colaneri, ``{Families of moment matching
  based, low order approximations for linear systems},'' \emph{Systems and
  Control Letters}, vol.~64, no.~1, pp. 47--56, 2014.

\bibitem{cheng-i-iftime-necoara-CDC2024}
X.~Cheng, T.~C. Ionescu, O.~V. Iftime, and I.~Necoara, ``Moment matching for
  second-order systems with pole-zero placement,'' in \emph{Proceedings of 63rd
  IEEE CDC}.\hskip 1em plus 0.5em minus 0.4em\relax IEEE, 2024, pp. 2170--2175.

\bibitem{astolfi-TAC2010}
A.~Astolfi, ``Model reduction by moment matching for linear and nonlinear
  systems,'' \emph{IEEE Trans. Autom. Contr.}, vol.~50, no.~10, pp. 2321--2336,
  2010.

\bibitem{i-astolfi-colaneri-SCL2014}
T.~C. Ionescu, A.~Astolfi, and P.~Colaneri, ``Families of moment matching
  based, low order approximations for linear systems,'' \emph{Systems \&
  Control Letters}, vol.~64, pp. 47--56, 2014.

\bibitem{byrnes1998convex}
C.~I. Byrnes, S.~V. Gusev, and A.~Lindquist, ``A convex optimization approach
  to the rational covariance extension problem,'' \emph{SIAM J. Control \&
  Optimization}, vol.~37, no.~1, pp. 211--229, 1998.

\bibitem{byrnes-lindquist-SIAM2008}
C.~I. Byrnes and A.~Lindquist, ``Important moments in systems and control,''
  \emph{SIAM J. Cont. \& Opt.}, vol.~47, no.~5, pp. 2458--2469, 2008.

\bibitem{antoulas-2005}
A.~C. Antoulas, \emph{Approximation of Large-Scale Dynamical Systems}.\hskip
  1em plus 0.5em minus 0.4em\relax Philadelphia: SIAM, 2005.

\bibitem{desouza-bhattacharyya-LAA1981}
E.~de~Souza and S.~P. Bhattacharyya, ``Controllability, observability and the
  solution of ${AX}-{XB}={C}$,'' \emph{Linear Algebra \& Its App.}, vol.~39,
  pp. 167--188, 1981.

\bibitem{bernstein1995second}
D.~S. Bernstein and S.~P. Bhat, ``Lyapunov stability, semistability, and
  asymptotic stability of matrix second-order systems,'' \emph{Journal of
  Mechanical Design}, vol. 117, no.~B, pp. 145--153, 1995.

\bibitem{scarciotti-CDC2015generators}
G.~Scarciotti and A.~Astolfi, ``Characterization of the moments of a linear
  system driven by explicit signal generators,'' in \emph{Proceedings of the
  2015 American Control Conference (ACC)}.\hskip 1em plus 0.5em minus
  0.4em\relax IEEE, 2015, pp. 589--594.

\bibitem{padoan2017geometric}
A.~Padoan, G.~Scarciotti, and A.~Astolfi, ``A geometric characterization of the
  persistence of excitation condition for the solutions of autonomous
  systems,'' \emph{IEEE Tr. on Aut. Cont.}, vol.~62, no.~11, pp. 5666--5677,
  2017.

\bibitem{Datta2012poleplacement}
S.~Datta, D.~Chakraborty, and B.~Chaudhuri, ``Partial pole placement with
  controller optimization,'' \emph{IEEE Transactions on Automatic Control},
  vol.~57, no.~4, pp. 1051--1056, 2012.

\bibitem{Ionescu2021poleplacement}
T.~C. Ionescu, O.~V. Iftime, and I.~Necoara, ``Model reduction with pole-zero
  placement and high order moment matching,'' \emph{Automatica}, vol. 138, p.
  110140, 2022.

\bibitem{astrom-murray-2008}
K.~A. {\AA}str\"om and R.~M. Murray, \emph{Feedback Systems: An Introduction
  for Scientists and Engineers}.\hskip 1em plus 0.5em minus 0.4em\relax
  Princeton Univ. Press, 2012, version v2.11b, electronic.

\bibitem{gugercin-antoulas-beattie-SIAM2008}
S.~Gugercin, A.~C. Antoulas, and C.~A. Beattie, ``${H}_2$ model reduction for
  large-scale dynamical systems,'' \emph{SIAM J. Matrix Analysis \& App.},
  vol.~30, no.~2, pp. 609--638, 2008.

\bibitem{simard-moreschini-astolfi-CDC2023}
J.~Simard, A.~Moreschini, and A.~Astolfi, ``Moment matching for nonlinear
  systems of second-order equations,'' in \emph{Proc. of 62nd Conference on
  Decision and Control}, 2023, pp. 4978--4983.

\end{thebibliography}

		\end{document}